\documentclass[11pt]{amsart}
\usepackage{amsmath,amsfonts,latexsym,graphicx,amssymb,url}
\usepackage{amsthm,txfonts,pifont,bbding,pxfonts,manfnt,mathrsfs}
\usepackage[dvipsnames]{xcolor}

\usepackage{amsmath,amsfonts,latexsym,graphicx,amssymb,url}
\usepackage{hyperref,pdfsync}
\usepackage{amsmath,txfonts,pifont,bbding,pxfonts,manfnt}
\usepackage[active]{srcltx}
\usepackage{wasysym,pstricks}
\usepackage{wrapfig, subfigure,graphicx}
\usepackage{hyperref}
\usepackage{pdfsync}
\usepackage{enumitem}

\usepackage{tikz}
\usetikzlibrary{arrows.meta}
\usetikzlibrary{datavisualization.formats.functions}
\usepackage{pgfplots}
\usetikzlibrary{shadows,shapes,positioning}

\usepackage{tikz-3dplot,graphicx,amsmath}
\usetikzlibrary{arrows.meta}

\usepackage[margin=1in]{geometry}
\usepackage{pdflscape}
\usepackage{flafter}
\usepackage{placeins}

\newcommand{\E}{\mathbf{E}}
\renewcommand{\H}{\mathbf{H}}
\newcommand{\D}{\mathbf{D}}
\newcommand{\B}{\mathbf{B}}
\newcommand{\G}{\mathbf{G}}

\newcommand{\C}{{\mathbb C}}

\newcommand{\R}{{\mathbb R}} 

\newcommand{\e}{\varepsilon}

\renewcommand{\(}{\left(}
\renewcommand{\)}{\right)}
\renewcommand{\k}{{\mathbf{k}}}
\renewcommand{\o}{\omega}
\newcommand{\m}{{\mathbf{m}}}

\renewcommand{\t}{{\rm{t}}}

\numberwithin{equation}{section}

\newtheorem{prop}{Proposition}[section]
\newtheorem{remark}{Remark}[section]

\newtheorem{theorem}{Theorem}[section]

\newtheorem{lemma}[theorem]{Lemma}

\newcommand{\cristiancomment}[1]{#1}
\newcommand{\ericcomment}[1]{#1}

\begin{document}

\title{Refraction laws in spatio-temporal media
}

\author{Cristian E. Guti\'errez}
\address[Cristian E. Guti\'errez]{Department of Mathematics, Temple University}
\email{gutierre@temple.edu}

\author{Eric Stachura}
\address[Eric Stachura]{Department of Mathematics, Kennesaw State University}
\email{estachur@kennesaw.edu}

\thanks{ \today}

\begin{abstract}
We study the time-dependent Maxwell system, formulated in the sense of distributions, for electromagnetic waves propagating through media with temporal and spatial material interfaces. Under explicit trace and regularity assumptions on the permittivity and permeability, we derive the jump conditions produced by temporal discontinuities and by subsequent spatial interfaces. These conditions are used to obtain generalized Snell laws for the wave vectors generated by temporal splitting, reflection, and transmission. We also derive the associated amplitude equations and organize them as a finite-dimensional linear system for a space-time slab geometry. Finally, we provide an explicit oblique-incidence example. 
Our approach does not impose a smooth-field ansatz and allows material parameters that need not be constant away from the interfaces.
\end{abstract}

\maketitle


\tableofcontents

\setcounter{equation}{0}
\section{Introduction}

Time-varying materials have recently attracted substantial attention because temporal modulation provides an additional mechanism for controlling electromagnetic waves. This point of view underlies applications such as frequency conversion, wave amplification, nonreciprocal propagation, and photonic time crystals; see, for example, \cite{mostafa2024temporal, lustig2023photonic, pacheco2025temporal}. For recent tutorials and surveys, see for instance \cite{gao2025fundamentals, zeng2025performance}, and for related developments for spatio-temporal platforms, we mention in particular \cite{pacheco2024spatiotemporal, antyufeyeva2025emulating}. These works build on the classical observation of Morgenthaler \cite{morgenthaler1958velocity} that a temporal discontinuity in the material parameters changes the frequency content of an electromagnetic wave and generates both forward and backward temporal branches.

A temporal interface occurs when one or more material parameters are changed rapidly while the optical wave is present in the medium \cite{engheta2023four}; physically, such a change can be produced, for instance, by a strong optical nonlinearity driven by an ultrafast laser pulse \cite{tirole2024second}. At a purely temporal interface, the spatial phase is matched across the jump, while the frequencies and amplitudes of the generated waves change according to the new material parameters \cite{mendoncca2002time, mendoncca2024time}. 
In our recent paper \cite{CGES25temporal}, we derived the generalized Snell's law at temporal interfaces directly from Maxwell's equations. Since the fields generated at such interfaces are generally discontinuous, the derivation was carried out in the framework of distributions, which provides a natural way to formulate the Maxwell system across temporal jumps. 

In this paper we focus on temporal interfaces and on their interaction with spatial interfaces. We use the geometric notation introduced more fully in Section \ref{sec:general_formulas}. Let
\[
\Omega=\Omega_1\cup\Gamma\cup\Omega_2,
\]
where the disjoint open sets $\Omega_1$ and $\Omega_2$ are separated by a smooth spatial interface $\Gamma$. At time $t=t_0$, the material changes throughout $\Omega$: before the jump we write the material parameters as $(\e^-,\mu^-)$, while after the jump they are $(\e_i^+,\mu_i^+)$ in $\Omega_i$, $i=1,2$. More generally, a superscript $-$ denotes a quantity for $t<t_0$, and a superscript $+$ together with a subscript $i$ denotes its value in $\Omega_i$ for $t>t_0$. Figure \ref{fig:intro-space-time-geometry} shows this basic geometry without introducing the individual reflected and transmitted branches, which are defined later.

\begin{figure}[!htbp]
\centering
\begin{tikzpicture}[
  >=Latex,
  line cap=round,
  line join=round,
  every node/.style={font=\small,text=black}
]
\fill[gray!14] (-5,-2.2) rectangle (5,0);
\fill[blue!12] (-5,0) rectangle (0,3);
\fill[ForestGreen!12] (0,0) rectangle (5,3);

\draw[->,thick] (-5.3,-2.45) -- (5.35,-2.45) node[right] {space};
\draw[->,thick] (-5.3,-2.45) -- (-5.3,3.25) node[above] {time};
\draw[very thick,blue!65!black] (-5,0) -- (5,0)
  node[pos=0.78,below=4pt] {temporal interface $t=t_0$};
\draw[very thick,orange!80!black] (0,0) -- (0,3)
  node[above] {spatial interface $\Gamma$};
\draw[dashed,gray] (0,-2.2) -- (0,0);

\node[align=center] at (0,-1.1)
  {$\Omega$, $t<t_0$\\material $(\e^-,\mu^-)$};
\node[align=center] at (-2.5,1.5)
  {$\Omega_1$, $t>t_0$\\material $(\e_1^+,\mu_1^+)$};
\node[align=center] at (2.5,1.5)
  {$\Omega_2$, $t>t_0$\\material $(\e_2^+,\mu_2^+)$};
\end{tikzpicture}
\caption{Basic space-time geometry. The material changes at the temporal interface $t=t_0$; after this change, the spatial interface $\Gamma$ separates the regions $\Omega_1$ and $\Omega_2$.}
\label{fig:intro-space-time-geometry}
\end{figure}
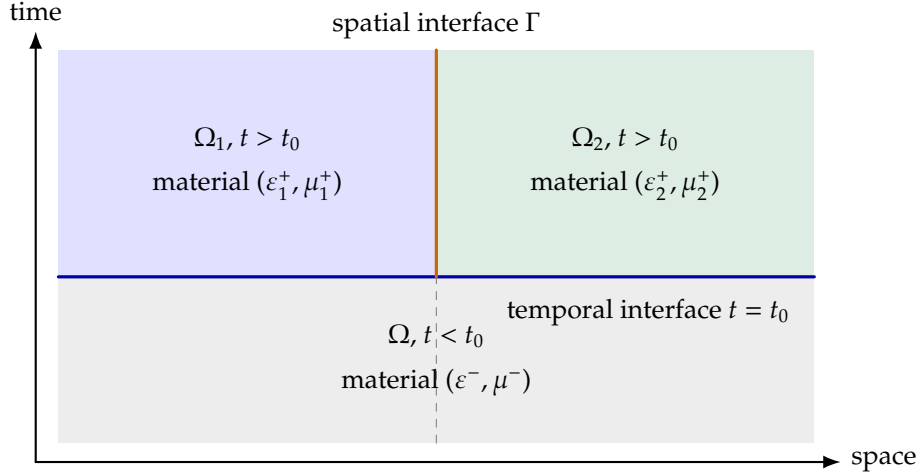
\FloatBarrier

Although the physical literature on time-varying media is now extensive, a distributional derivation of the interface conditions is useful because the relevant fields and material parameters are naturally discontinuous. We therefore formulate the time-dependent Maxwell system in the sense of distributions and derive the jump conditions directly from Maxwell's equations, rather than imposing them through a smoothness ansatz. This viewpoint is close in spirit to the distributional treatment of spatial metasurfaces in \cite{2025-gutierrez-sabra:Maxwelleqsandgeneralizedsnelllaw}. It also makes clear which traces of the fields and material parameters are needed at temporal and spatial interfaces.

The main goal is then to use these distributional boundary conditions to derive generalized Snell laws and the associated amplitude equations for a medium with both temporal and spatial jumps.   Explicit formulas for these Snell laws are useful in practice because they translate prescribed material jumps into explicit predictions for outgoing frequencies, propagation directions, and amplitudes, which are the quantities measured and engineered in time-varying optical devices. They also make it possible to compare different slab designs and identify parameter regimes in which a desired reflected or transmitted branch is enhanced or suppressed. The phase relations determine the possible wave vectors, while the electric and magnetic boundary conditions determine, or impose compatibility conditions on, the corresponding amplitudes. In the final amplitude calculation, the problem becomes a finite-dimensional linear system whose solvability depends on the prescribed incident field and on the material constants in the slabs.

The outline of this paper is as follows. In Section \ref{sec:Maxwell_distributions}, we formulate the Maxwell system in the sense of space-time distributions and collect the distributional identities needed later, including formulas for time derivatives, divergence, and curl across interfaces. In Section \ref{sec:boundary conditions}, we apply these identities to spatio-temporal material parameters and derive the temporal jump conditions for the electric and magnetic fields, see Theorem \ref{thm:temporal boundary conditions}. In Section \ref{sec:GSL_spacetime}, we use these boundary conditions to derive generalized Snell laws, after first recalling a useful exponential lemma used to separate distinct phase factors. The temporal interface is treated in Section \ref{sec:temporal_interface}, while the spatial interface is treated in Section \ref{sec:spatial_interface}, culminating in Theorem \ref{thm:GSL_spatial_interface}. In Section \ref{sec:amplitude_calculation}, we then use the phase vectors from these Snell laws to calculate the field amplitudes in a space-time slab geometry. There, we first organize the generated fields, then derive the electric and magnetic field amplitude equations, assemble the full linear system \eqref{eq:amplitude-system}, record its solvability conditions, and work out an explicit oblique-incidence example.

\FloatBarrier
\setcounter{equation}{0}
\section{Maxwell equations in the sense of distributions}\label{sec:Maxwell_distributions}

Recall the Maxwell system of equations (in CGS units)
\begin{align}
&\nabla \cdot \D = 4\pi \rho, \label{divergence E zero}\tag{M.1}\\
&\nabla \cdot \B = 0, \label{divergence B zero}\tag{M.2}\\
&\nabla \times \E = -\dfrac{1}{c}\dfrac{\partial \B}{\partial t}\label{Faraday law}\tag{M.3}\\
&\nabla \times \H= \dfrac{4\pi}{c}\mathbf{J}+\dfrac{1}{c}\dfrac{\partial \D}{\partial t}. \label{Ampere Maxwell}\tag{M.4}
\end{align}
where $\E(x,t)$ denotes the electric field, $\H(x,t)$ the magnetic field, $\D(x,t)$ the electric flux density, $\B(x,t)$ the magnetic flux density, and $c$ denotes the speed of light. See \cite[Sections 1.1-1.2]{BornWolf}.
We assume the charge density $\rho=0$ and the current density $\mathbf{J}=0$, and
consider the constitutive equations given by 
\begin{equation}\label{eq:constitutive}
\D=\e\,\E,\qquad \B=\mu\,\H
\end{equation}
where we assume that $\e=\e(x,t)$ and $\mu=\mu(x,t)$ are positive functions of $x\in \R^3,t\in \R$, $t\neq t_0$ for some $t_0$, that satisfy appropriate piecewise smoothness assumptions that will be described in a moment.

We begin recalling the notions needed to analyze the Maxwell system in distributional sense.
Let $\Omega\subset \R^3$ be an open domain, that could be the whole space, and let $(a,b)\subset \R$ be an open interval that could be infinite. 
A generalized function or distribution defined in $\Omega\times (a,b)$ is a complex-valued continuous linear functional defined on the class of test functions denoted by $\mathcal D(\Omega\times (a,b))=C_0^\infty(\Omega\times (a,b))$, that is, the class of functions are infinitely differentiable in $\Omega\times (a,b)$ and have compact support in $\Omega\times (a,b)$. 
More precisely, $g$ is a distribution in $\Omega\times (a,b)$ if $g:\mathcal D(\Omega\times (a,b))\to \C$ is a linear function such that for each compact $K\subset \Omega\times (a,b)$ there exist a constant $C$ and an integer $k$ such that 
\begin{equation*}\label{eq:distribution estimate}
|g(\phi)|\leq C\,\sum_{|\alpha|\leq k}\sup_{x\in K}|D^\alpha\phi(x)|
\end{equation*}
for each $\phi\in \mathcal D(\Omega\times (a,b))$.
%
As customary, $\mathcal D’(\Omega\times (a,b))$ denotes the class of distributions in $\Omega\times (a,b)$. 

Numerous references are available for distributions; however, we only cite the seminal work by L. Schwartz, \cite{schwartz:theoriedesdistributions}, and also \cite{folland1999real}. 

If $g\in \mathcal D'(\Omega\times (a,b))$, then $\langle g,\phi \rangle$ denotes the value of the distribution $g$ on the test function $\phi\in \mathcal D(\Omega\times (a,b))$.
If $g$ is a locally integrable function in $\Omega\times (a,b)$, then $g$ gives rise to a distribution defined by 
\[
\langle g,\phi\rangle=\int_{\Omega\times (a,b)} g(x,t)\,\phi(x,t)\,dxdt,
\]
for each $\phi\in \mathcal D(\Omega\times (a,b))$.

We say that ${\bf G}=(G_1,G_2,G_3)$ is a vector valued distribution in $\Omega\times (a,b)$ if each component $G_i\in \mathcal D'(\Omega\times (a,b))$, $1\leq i\leq 3$.
The divergence of ${\bf G}$ with respect to $x$ is the scalar distribution defined by
\begin{equation}\label{form:divergence}
\langle \nabla \cdot {\bf G}, \phi\rangle
=-\sum_{i=1}^3 \langle G_i,\partial_{x_i} \phi\rangle,
\end{equation}
and the curl of ${\bf G}$ is the vector valued distribution in $\Omega$ defined by
\begin{equation}\label{form:curl bis}
\langle \nabla \times {\bf G},\phi\rangle=\left(\langle G_2,\phi_{x_3}\rangle-\langle G_3,\phi_{x_2}\rangle\right){\bf i}-\left(\langle G_1,\phi_{x_3}\rangle-\langle G_3,\phi_{x_1}\rangle\right){\bf j}+\left(\langle G_1,\phi_{x_2}\rangle-\langle G_2,\phi_{x_1}\rangle\right){\bf k}.
\end{equation}
Then it follows that
\begin{equation}\label{eq:div of curl zero}
\nabla \cdot (\nabla\times {\bf G})=0,
\end{equation} 
in the sense of distributions.
When the distribution ${\bf G}=(G_1,G_2,G_3)$ is given by a locally integrable function in $\Omega\times (a,b)$ we obtain from \eqref{form:divergence}, and \eqref{form:curl bis} that 
\begin{equation*}
\langle \nabla \cdot {\bf G},\phi\rangle=-\int_{\Omega\times (a,b)}{\bf G}\cdot \nabla_x \phi\,dxdt,\qquad
\langle \nabla \times {\bf G},\phi\rangle=\int_{\Omega\times (a,b)} {\bf G}\times \nabla_x \phi\,dxdt.
\end{equation*}
The derivative of $\G$ with respect to $t$ in the sense of distributions is by definition the distribution $\dfrac{\partial \G}{\partial t}$ defined by
\[
\left\langle \dfrac{\partial \G}{\partial t},\phi \right\rangle
=
-\left\langle \G, \dfrac{\partial \phi}{\partial t}\ \right\rangle,
\]
for each $\phi\in \mathcal D(\Omega\times (a,b))$.

Therefore, the fields $\D,\B,\E$ and $\H$ solve the Maxwell system in the sense of distributions if 
$\D,\B,\E$ and $\H$ are vector valued distributions in $\Omega\times (a,b)$ that satisfy the equations \eqref{divergence E zero}, \eqref{divergence B zero}, \eqref{Faraday law}, and \eqref{Ampere Maxwell} in the sense of distributions.

{\color{black}
\subsection{General formulas}\label{sec:general_formulas}
Since our application to spatio-temporal interfaces requires considering distributions represented by discontinuous functions, we need a comprehensive set of representation formulas, which will be used in Section \ref{sec:boundary conditions} to determine the appropriate boundary conditions.

Let $\Omega\subset \R^3$ be an open domain, possibly all of $\R^3$. Let $t_0\in (a,b)
\subset \R$, where $(a,b)$ may be unbounded.
Suppose $\Omega=\Omega_1\cup \Gamma \cup \Omega_2$, where the sets $\Omega_i$ are open, connected, and disjoint, and $\Gamma$ is a smooth surface separating $\Omega_1$ and $\Omega_2$.
Define a field $\G(x,t)$ for $x\in \Omega$ and $t\neq t_0$ by
\begin{equation*}
\G(x,t)
=
\begin{cases}
\G^-(x,t) &\text{for $x\in \Omega$ and $t<t_0$}\\
\G^+_1(x,t) &\text{for $x\in \Omega_1$ and $t>t_0$}\\
\G^+_2(x,t) &\text{for $x\in \Omega_2$ and $t>t_0$}
\end{cases}
\end{equation*}
and assume that it satisfies the following conditions:
%
\begin{enumerate}
\item\label{eq:G is C1} $\G^-\in C^1\(\Omega\times (a,t_0)\)$, $\G^+_i\in C^1\(\Omega_i\times (t_0,b)\)$, and $\G\in L^1_{\text{loc}}\(\Omega\times (a,b)\)$\footnote{As usual, $L^1_{\text{loc}}\(\Omega\times (a,b)\)$ denotes the class of complex-valued functions that are locally Lebesgue integrable in $\Omega\times (a,b)$.};
\item\label{eq:limits are finite} the following limits exist and are finite:
\[
\lim_{t\to t_0^-}\G^-(x,t):=\G^-(x,t_0) \quad \text{for all } x\in \Omega,
\]
\[
\lim_{t\to t_0^+}\G^+_i(x,t):=\G^+_i(x,t_0) \quad \text{for all } x\in \Omega_i,\ i=1,2,
\]
\[
\lim_{\substack{x\to \bar x\\ x\in \Omega_i}}\G^+_i(x,t):=\G^+_i(\bar x,t) \quad \text{for all } \bar x\in \Gamma,\ t>t_0,\ i=1,2,
\]
and
\[
\lim_{\substack{t\to t_0^+\\ x\to \bar x\\ x\in \Omega_i}}\G^+_i(x,t):=\G^+_i(\bar x,t_0) \quad \text{for all } \bar x\in \Gamma,\ i=1,2;
\]
\item\label{eq:derivatives if G wrt t are integrable} $\dfrac{\partial \G}{\partial t}$ and $\nabla_x \G$ are locally integrable in $\Omega\times (a,b)$.
\end{enumerate}
For $\phi\in C_0^\infty \(\Omega\times (a,b)\)$, define the linear functional
\[
\langle \G,\phi\rangle=\int_\Omega \int_a^{t_0} \G^-(x,t)\phi(x,t)\,dx\,dt
+\int_{\Omega_1} \int_{t_0}^b \G^+_1(x,t)\phi(x,t)\,dx\,dt
+\int_{\Omega_2} \int_{t_0}^b \G^+_2(x,t)\phi(x,t)\,dx\,dt
\]
This functional defines a distribution in $\Omega\times (a,b)$.



The following proposition gives a representation formula for the distributional time derivative $\dfrac{\partial \G}{\partial t}$.
\begin{prop}
\label{prop:timederivativeforGfortwomedia}
For each $\phi\in C_0^1(\Omega\times (a,b))$, we have
\begin{align*}
\left\langle \dfrac{\partial \G}{\partial t},\phi \right\rangle
&=
\int_{\Omega_1} \(\G_1^+(x,t_0^+)-\G^-(x,t_0^-)\)\phi(x,t_0)\,dx
+\int_{\Omega_2} \(\G_2^+(x,t_0^+)-\G^-(x,t_0^-)\)\phi(x,t_0)\,dx \\
&\qquad +\int_{\Omega} \int_a^{t_0} \partial_t\G^-(x,t)\phi(x,t)\,dt\,dx \\
&\qquad +\int_{\Omega_1}\int_{t_0}^b \partial_t\G_1^+(x,t)\phi(x,t)\,dt\,dx
+\int_{\Omega_2} \int_{t_0}^b \partial_t\G_2^+(x,t)\phi(x,t)\,dt\,dx.
\end{align*}
\end{prop}
\begin{proof}
By definition,
\[
\langle \G,\phi_t\rangle=\int_\Omega \int_a^{t_0} \G^-(x,t)\phi_t(x,t)\,dx\,dt
+\int_{\Omega_1} \int_{t_0}^b \G^+_1(x,t)\phi_t(x,t)\,dx\,dt
+\int_{\Omega_2} \int_{t_0}^b \G^+_2(x,t)\phi_t(x,t)\,dx\,dt.
\]
By assumptions \eqref{eq:G is C1} and \eqref{eq:derivatives if G wrt t are integrable}, integration by parts in $t$ gives, for each fixed $x$,
\begin{align*}
\int_a^{t_0} \G^-(x,t)\phi_t(x,t)\,dt
&=
\G^-(x,t)\phi(x,t)|_{t=a}^{t=t_0^-}-\int_a^{t_0} \partial_t\G^-(x,t)\phi(x,t)\,dt\\
&=
\G^-(x,t_0^-)\phi(x,t_0)-\int_a^{t_0} \partial_t\G^-(x,t)\phi(x,t)\,dt,
\end{align*}
\begin{align*}
\int_{t_0}^b \G_1^+(x,t)\phi_t(x,t)\,dt
&=
\G_1^+(x,t)\phi(x,t)|_{t_0^+}^{b}-\int_{t_0}^b \partial_t\G_1^+(x,t)\phi(x,t)\,dt\\
&=
-\G_1^+(x,t_0^+)\phi(x,t_0)-\int_{t_0}^b \partial_t\G_1^+(x,t)\phi(x,t)\,dt,
\end{align*}
and
\begin{align*}
\int_{t_0}^b \G_2^+(x,t)\phi_t(x,t)\,dt
&=
\G_2^+(x,t)\phi(x,t)|_{t_0^+}^{b}-\int_{t_0}^b \partial_t\G_2^+(x,t)\phi(x,t)\,dt\\
&=
-\G_2^+(x,t_0^+)\phi(x,t_0)-\int_{t_0}^b \partial_t\G_2^+(x,t)\phi(x,t)\,dt.
\end{align*}
Here the endpoint terms at $a$ and $b$ vanish because $\phi$ has compact support in $\Omega\times(a,b)$. Therefore,
\begin{align*}
\langle \G,\phi_t\rangle
&=\int_\Omega \int_a^{t_0} \G^-(x,t)\phi_t(x,t)\,dx\,dt
+\int_{\Omega_1} \int_{t_0}^b \G^+_1(x,t)\phi_t(x,t)\,dx\,dt
+\int_{\Omega_2} \int_{t_0}^b \G^+_2(x,t)\phi_t(x,t)\,dx\,dt\\
&=\int_\Omega \(\G^-(x,t_0^-)\phi(x,t_0)-\int_a^{t_0} \partial_t\G^-(x,t)\phi(x,t)\,dt\)\,dx\\
&\qquad +\int_{\Omega_1} \(-\G_1^+(x,t_0^+)\phi(x,t_0)-\int_{t_0}^b \partial_t\G_1^+(x,t)\phi(x,t)\,dt\)\,dx\\
&\qquad \qquad 
+\int_{\Omega_2} \(-\G_2^+(x,t_0^+)\phi(x,t_0)-\int_{t_0}^b \partial_t\G_2^+(x,t)\phi(x,t)\,dt\)\,dx\\
&=\int_{\Omega} \G^-(x,t_0^-)\phi(x,t_0)\,dx-\int_{\Omega} \int_a^{t_0} \partial_t\G^-(x,t)\phi(x,t)\,dt\,dx\\
&\qquad +\int_{\Omega_1} \(-\G_1^+(x,t_0^+)\phi(x,t_0)-\int_{t_0}^b \partial_t\G_1^+(x,t)\phi(x,t)\,dt\)\,dx\\
&\qquad \qquad 
+\int_{\Omega_2} \(-\G_2^+(x,t_0^+)\phi(x,t_0)-\int_{t_0}^b \partial_t\G_2^+(x,t)\phi(x,t)\,dt\)\,dx\\
&=\int_{\Omega_1} \G^-(x,t_0^-)\phi(x,t_0)\,dx +\int_{\Omega_2} \G^-(x,t_0^-)\phi(x,t_0)\,dx -\int_{\Omega} \int_a^{t_0} \partial_t\G^-(x,t)\phi(x,t)\,dt\,dx\\
&\qquad +\int_{\Omega_1} \(-\G_1^+(x,t_0^+)\phi(x,t_0)-\int_{t_0}^b \partial_t\G_1^+(x,t)\phi(x,t)\,dt\)\,dx\\
&\qquad \qquad 
+\int_{\Omega_2} \(-\G_2^+(x,t_0^+)\phi(x,t_0)-\int_{t_0}^b \partial_t\G_2^+(x,t)\phi(x,t)\,dt\)\,dx\\
&=\int_{\Omega_1} \(\G^-(x,t_0^-)-\G_1^+(x,t_0^+)\)\phi(x,t_0)\,dx\\
&\qquad  +\int_{\Omega_2} \(\G^-(x,t_0^-)-\G_2^+(x,t_0^+)\)\phi(x,t_0)\,dx -\int_{\Omega} \int_a^{t_0} \partial_t\G^-(x,t)\phi(x,t)\,dt\,dx\\
&\qquad -\int_{\Omega_1}\int_{t_0}^b \partial_t\G_1^+(x,t)\phi(x,t)\,dt\,dx
-\int_{\Omega_2} \int_{t_0}^b \partial_t\G_2^+(x,t)\phi(x,t)\,dt\,dx.
\end{align*}
Since $\left\langle \partial_t\G,\phi\right\rangle=-\langle \G,\phi_t\rangle$, this gives the desired formula.

\end{proof}


The next proposition gives a representation formula for the distributional divergence of $\G$.
\begin{prop}\label{prop:divergence_formula}
For each $\phi\in C_0^1(\Omega\times (a,b))$, we have
\begin{align*}
\left\langle \nabla_x \cdot \G,\phi \right\rangle
&=
\int_{t_0}^b  \int_{\Gamma} \phi(x,t) \(\G_2^+(x,t)-\G_1^+(x,t)\) \cdot \mathbf{n}(x)\,d\sigma(x)\,dt
+\int_a^{t_0} \int_{\Omega} \phi\, \nabla_x \cdot \G^-\,dx\,dt\\
&\qquad \qquad +\int_{t_0}^b \int_{\Omega_1} \phi\, \nabla_x \cdot \G_1^+\,dx\,dt
+\int_{t_0}^b \int_{\Omega_2} \phi\, \nabla_x \cdot \G_2^+\,dx\,dt,
\end{align*}
where, for $x\in \Gamma$ and $t>t_0$,
\[
\G_2^+(x,t)-\G_1^+(x,t)
=\lim_{\substack{y\to x\\ y\in \Omega_2}}\G^+_2(y,t)-\lim_{\substack{y\to x\\ y\in \Omega_1}}\G^+_1(y,t),
\]
and $\mathbf{n}(x)$ denotes the unit normal to $\Gamma$ at $x$, oriented from $\Omega_1$ to $\Omega_2$.
\end{prop}
\begin{proof}
Recall that
\[
\langle \G,\phi\rangle=\int_\Omega \int_a^{t_0} \G^-(x,t)\phi(x,t)\,dx\,dt
+\int_{\Omega_1} \int_{t_0}^b \G^+_1(x,t)\phi(x,t)\,dx\,dt
+\int_{\Omega_2} \int_{t_0}^b \G^+_2(x,t)\phi(x,t)\,dx\,dt,
\]
and, by the definition of distributional divergence,
\begin{align*}
\left\langle \nabla_x \cdot \G,\phi \right\rangle
&=
-
\sum_{i=1}^3
\int_\Omega \int_a^{t_0} G_i^-(x,t)\,\dfrac{\partial \phi(x,t)}{\partial x_i}\,dx\,dt\\
&\qquad -\sum_{i=1}^3
\int_{\Omega_1} \int_{t_0}^b G_{i 1}^+(x,t)\,\dfrac{\partial \phi(x,t)}{\partial x_i}\,dx\,dt-\sum_{i=1}^3
\int_{\Omega_2} \int_{t_0}^b G_{i 2}^+(x,t)\,\dfrac{\partial \phi(x,t)}{\partial x_i}\,dx\,dt\\
&=-
\int_\Omega \int_a^{t_0} \G^-(x,t)\cdot \nabla_x \phi(x,t)\,dx\,dt
-\int_{\Omega_1} \int_{t_0}^b \G_1^+(x,t)\cdot \nabla_x \phi(x,t)\,dx\,dt\\
&\qquad 
-\int_{\Omega_2} \int_{t_0}^b \G_2^+(x,t)\cdot \nabla_x \phi(x,t)\,dx\,dt\\
&=-
\int_a^{t_0} \int_\Omega  \G^-(x,t)\cdot \nabla_x \phi(x,t)\,dx\,dt
-\int_{t_0}^b \int_{\Omega_1}  \G_1^+(x,t)\cdot \nabla_x \phi(x,t)\,dx\,dt\\
&\qquad 
-\int_{t_0}^b\int_{\Omega_2}  \G_2^+(x,t)\cdot \nabla_x \phi(x,t)\,dx\,dt.
\end{align*}
Using the divergence theorem for each fixed $t\neq t_0$ yields
\begin{align*}
\int_\Omega  \G^-(x,t)\cdot \nabla_x \phi(x,t)\,dx
&=
\int_\Omega \(\nabla_x\cdot  (\phi \G^-) - \phi \nabla_x \cdot \G^-\)\,dx\\
&=
\int_{\partial \Omega} \phi \G^-\cdot \mathbf{n}\,d\sigma
-\int_{\Omega} \phi \nabla_x \cdot \G^-\,dx
=-\int_{\Omega} \phi \nabla_x \cdot \G^-\,dx,
\end{align*}
since $\phi(\cdot,t)$ has compact support contained in $\Omega$.
Similarly,
\begin{align*}
\int_{\Omega_1}  \G_1^+(x,t)\cdot \nabla_x \phi(x,t)\,dx
&=
\int_{\Omega_1} \(\nabla_x\cdot  (\phi \G_1^+) - \phi \nabla_x \cdot \G_1^+\)\,dx\\
&=
\int_{\Gamma} \phi \G_1^+\cdot \mathbf{n}_1\,d\sigma
-\int_{\Omega_1} \phi \nabla_x \cdot \G_1^+\,dx
\end{align*}
where $\mathbf{n}_1$ denotes the unit normal to $\Gamma$ oriented from $\Omega_1$ to $\Omega_2$,
and
\begin{align*}
\int_{\Omega_2}  \G_2^+(x,t)\cdot \nabla_x \phi(x,t)\,dx
&=
\int_{\Omega_2} \(\nabla_x\cdot  (\phi \G_2^+) - \phi \nabla_x \cdot \G_2^+\)\,dx\\
&=
\int_{\Gamma} \phi \G_2^+\cdot \mathbf{n}_2\,d\sigma
-\int_{\Omega_2} \phi \nabla_x \cdot \G_2^+\,dx,
\end{align*}
where $\mathbf{n}_2$ denotes the unit normal to $\Gamma$ oriented from $\Omega_2$ to $\Omega_1$; hence $\mathbf{n}_2=-\mathbf{n}_1$.
Substituting these expressions into the formula for $\left\langle \nabla_x \cdot \G,\phi \right\rangle$, we obtain
\begin{align*}
\left\langle \nabla_x \cdot \G,\phi \right\rangle
&=
-
\int_a^{t_0} \int_\Omega  \G^-(x,t)\cdot \nabla_x \phi(x,t)\,dx\,dt
-\int_{t_0}^b \int_{\Omega_1}  \G_1^+(x,t)\cdot \nabla_x \phi(x,t)\,dx\,dt\\
&\qquad 
-\int_{t_0}^b\int_{\Omega_2}  \G_2^+(x,t)\cdot \nabla_x \phi(x,t)\,dx\,dt\\
&=
\int_a^{t_0} \int_{\Omega} \phi\, \nabla_x \cdot \G^-\,dx\,dt
+\int_{t_0}^b  \(\int_{\Gamma} \phi \(\G_2^+ -\G_1^+\) \cdot \mathbf{n}_1\,d\sigma
+\int_{\Omega_1} \phi \nabla_x \cdot \G_1^+\,dx\)dt\\
&\qquad 
+\int_{t_0}^b\(
\int_{\Omega_2} \phi \nabla_x \cdot \G_2^+\,dx\)dt\\
&=
\int_{t_0}^b  \int_{\Gamma} \phi \(\G_2^+ -\G_1^+\) \cdot \mathbf{n}_1\,d\sigma\,dt
+\int_a^{t_0} \int_{\Omega} \phi\, \nabla_x \cdot \G^-\,dx\,dt\\
&\qquad +\int_{t_0}^b \int_{\Omega_1} \phi\, \nabla_x \cdot \G_1^+\,dx\,dt
+\int_{t_0}^b
\int_{\Omega_2} \phi\, \nabla_x \cdot \G_2^+\,dx\,dt.
\end{align*}
This proves the desired formula.
\end{proof}
\begin{prop}\label{prop:formulaforcurlforgeneralmedia}
For each $\phi\in C_0^1(\Omega\times (a,b))$, we have the following representation formula for the distributional curl of $\G$:
\begin{align*}
\langle \nabla \times \G, \phi \rangle
&=
-\int_{t_0}^b \int_\Gamma \(\(\G_2^+-\G_1^+\)\times \mathbf{n}\)\,\phi\,d\sigma\,dt
+\int_a^{t_0} \int_{\Omega}\phi\,\nabla \times \G^- \,dx\,dt\\
&\qquad +\int_{t_0}^b\int_{\Omega_1}\phi\,\nabla \times\G_1^+ \,dx\,dt
+ \int_{t_0}^b\int_{\Omega_2}\phi\,\nabla \times\G_2^+ \,dx\,dt,
\end{align*}
where, for $x\in \Gamma$ and $t>t_0$,
\[
\G_2^+(x,t)-\G_1^+(x,t)
=\lim_{\substack{y\to x\\ y\in \Omega_2}}\G^+_2(y,t)-\lim_{\substack{y\to x\\ y\in \Omega_1}}\G^+_1(y,t),
\]
and $\mathbf{n}(x)$ denotes the unit normal to $\Gamma$ at $x$, oriented from $\Omega_1$ to $\Omega_2$.
\end{prop}
\begin{proof}
By assumption \eqref{eq:G is C1}, $\G$ is locally integrable in $\Omega\times (a,b)$. Hence, from \eqref{form:curl bis}, we can write
\begin{align*}
\langle \nabla \times \G, \phi \rangle
&=
\int_{\Omega\times (a,b)} \G\times \nabla \phi\,dx\,dt\\
&=
\int_\Omega \int_a^{t_0} \G^-(x,t)\times \nabla \phi(x,t)\,dx\,dt
+\int_{\Omega_1} \int_{t_0}^b \G^+_1(x,t)\times \nabla \phi(x,t)\,dx\,dt
+\int_{\Omega_2} \int_{t_0}^b \G^+_2(x,t)\times \nabla \phi(x,t)\,dx\,dt\\
&=A+B+C
\end{align*}
for each $\phi\in C^1_0\(\Omega\times (a,b)\)$.

We first compute $B$. Set $\G_1^+=\(G_{1,1}^+,G_{1,2}^+,G_{1,3}^+\)$. 
For fixed $t>t_0$, we write  
\begin{align*}
&\int_{\Omega_1}  \G^+_1(x,t)\times \nabla \phi(x,t)\,dx\\
&=
\left(\langle G_{1,2}^+,\phi_{x_3}\rangle-\langle G_{1,3}^+,\phi_{x_2}\rangle\right)\mathbf{i}
-\left(\langle G_{1,1}^+,\phi_{x_3}\rangle-\langle G_{1,3}^+,\phi_{x_1}\rangle\right)\mathbf{j}
+\left(\langle G_{1,1}^+,\phi_{x_2}\rangle-\langle G_{1,2}^+,\phi_{x_1}\rangle\right)\mathbf{k}.
\end{align*}
Now, for $i\neq j$,
\begin{align*}
\langle G_{1,i}^+,\phi_{x_j}\rangle-\langle G_{1,j}^+,\phi_{x_i}\rangle
&=
\int_{\Omega_1}
\(G_{1,i}^+(x,t)\phi_{x_j}(x,t)-G_{1,j}^+(x,t)\phi_{x_i}(x,t)\)\,dx\\
&=\int_{\Omega_1}
G_{1,i}^+(x,t)\phi_{x_j}(x,t)\,dx-\int_{\Omega_1}G_{1,j}^+(x,t)\phi_{x_i}(x,t)\,dx\\
&=
\int_{\Omega_1}
\(\(G_{1,i}^+\phi\)_{x_j}- \phi\,\(G_{1,i}^+\)_{x_j}\) \,dx- \int_{\Omega_1}\(\(G_{1,j}^+\phi\)_{x_i}- \phi\,\(G_{1,j}^+\)_{x_i}\) \,dx\\
&=
\int_{\Omega_1}
\(G_{1,i}^+\phi\)_{x_j}\,dx- \int_{\Omega_1}\phi\,\(G_{1,i}^+\)_{x_j} \,dx- 
\int_{\Omega_1}\(G_{1,j}^+\phi\)_{x_i}\,dx+ \int_{\Omega_1}\phi\,\(G_{1,j}^+\)_{x_i} \,dx\\
&=
\int_\Gamma \(G_{1,i}^+ \,n_j- G_{1,j}^+ \,n_i\)\,\phi\,d\sigma
+
\int_{\Omega_1}\phi\,\(\(G_{1,j}^+\)_{x_i}-\(G_{1,i}^+\)_{x_j}\) \,dx,
\end{align*}
by the divergence theorem, where $n_i$ are the components of $\mathbf{n}(x)$, the outer normal to $\Gamma$ oriented from $\Omega_1$ to $\Omega_2$, and $\G_1^+(x,t)=\lim_{\substack{y\to x\\ y\in \Omega_1}} \G_1^+(y,t)$ for $x\in \Gamma$.
Hence,
\begin{align*}
\int_{\Omega_1}  \G^+_1(x,t)\times \nabla \phi(x,t)\,dx
=
\int_\Gamma \(\G_1^+\times \mathbf{n}\)\,\phi\,d\sigma
+
\int_{\Omega_1}\phi\,\nabla \times\G_1^+ \,dx.
\end{align*}
The calculation for $C$ is analogous. For $\G_2^+$, we obtain
\begin{align*}
\int_{\Omega_2}  \G^+_2(x,t)\times \nabla \phi(x,t)\,dx
=
\int_\Gamma \(\G_2^+\times \mathbf{n}'\)\,\phi\,d\sigma
+
\int_{\Omega_2}\phi\,\nabla \times\G_2^+ \,dx
\end{align*}
where $\mathbf{n}'$ is the outer normal to $\Gamma$ oriented from $\Omega_2$ to $\Omega_1$, and $\G_2^+(x,t)=\lim_{\substack{y\to x\\ y\in \Omega_2}} \G_2^+(y,t)$ for $x\in \Gamma$.
Since $\mathbf{n}'=-\mathbf{n}$, we obtain 
\begin{align*}
B+C
&=
-\int_{t_0}^b \int_\Gamma \(\(\G_2^+-\G_1^+\)\times \mathbf{n}\)\,\phi\,d\sigma\,dt
+
\int_{t_0}^b\int_{\Omega_1}\phi\,\nabla \times\G_1^+ \,dx\,dt+ \int_{t_0}^b\int_{\Omega_2}\phi\,\nabla \times\G_2^+ \,dx\,dt.
\end{align*}
Similarly, we obtain the following expression for $A$:
\begin{align*}
A=\int_a^{t_0} \int_{\Omega}\phi\,\nabla \times\G^- \,dx\,dt
\end{align*}
since $\phi$ has compact support in $\Omega\times (a,b)$.
\end{proof}
}

\setcounter{equation}{0}
\section{Boundary conditions for spatio-temporal media}\label{sec:boundary conditions}


{\color{black}
We consider material parameters $\e$ and $\mu$ defined by
\begin{equation}\label{eps-mu-piecewise-def}
\e(x,t)
=
\begin{cases}
\e^-(x,t) &\text{for $x\in \Omega$ and $t<t_0$}\\
\e^+_1(x,t) &\text{for $x\in \Omega_1$ and $t>t_0$}\\
\e^+_2(x,t) &\text{for $x\in \Omega_2$ and $t>t_0$}
\end{cases}
\qquad 
\mu(x,t)
=
\begin{cases}
\mu^-(x,t) &\text{for $x\in \Omega$ and $t<t_0$}\\
\mu^+_1(x,t) &\text{for $x\in \Omega_1$ and $t>t_0$}\\
\mu^+_2(x,t) &\text{for $x\in \Omega_2$ and $t>t_0$.}
\end{cases}
\end{equation}
We assume that $\e$ satisfies the following conditions:
%
\begin{enumerate}
\item\label{eq:e is C1} $\e^-\in C^1\(\Omega\times (a,t_0)\)$, $\e^+_i\in C^1\(\Omega_i\times (t_0,b)\)$, and $\e\in L^1_{\text{loc}}\(\Omega\times (a,b)\)$\footnote{As usual, $L^1_{\text{loc}}\(\Omega\times (a,b)\)$ denotes the class of complex-valued functions that are locally Lebesgue integrable in $\Omega\times (a,b)$.}, $i=1,2$;
\item\label{eq:e-limits-are-finite} the following limits exist and are finite:
\[
\lim_{t\to t_0^-}\e^-(x,t):=\e^-(x,t_0) \quad \text{for all } x\in \Omega,
\]
\[
\lim_{t\to t_0^+}\e^+_i(x,t):=\e^+_i(x,t_0) \quad \text{for all } x\in \Omega_i,\ i=1,2,
\]
\[
\lim_{\substack{x\to y\\ x\in \Omega_i}}\e^+_i(x,t):=\e^+_i(y,t) \quad \text{for all } y\in \Gamma,\ t>t_0,\ i=1,2,
\]
and
\[
\lim_{\substack{t\to t_0^+\\ x\to y\\ x\in \Omega_i}}\e^+_i(x,t):=\e^+_i(y,t_0) \quad \text{for all } y\in \Gamma,\ i=1,2;
\]
\item\label{eq:e-derivatives-integrable} $\dfrac{\partial \e}{\partial t}$ and $\nabla_x \e$ are locally integrable in $\Omega\times (a,b)$;
\end{enumerate}
with analogous conditions imposed on $\mu$.
We assume that the electric field $\E$ has the form
\begin{equation*}
\E=
\begin{cases}
\E^- & \text{for $t<t_0$ and $x\in \Omega$}\\
\E^+_1 & \text{for $t>t_0$ and $x\in \Omega_1$}\\
\E^+_2 & \text{for $t>t_0$ and $x\in \Omega_2$}
\end{cases}
\end{equation*}
The magnetic field $\H$ is assumed to have an analogous form. 
For $y\in \Gamma$, we define the traces
\begin{align*}
\E_1^+(y,t)&=\lim_{\substack{x\to y\\ x\in \Omega_1}}\E_1^+(x,t),\\
\E_2^+(y,t)&=\lim_{\substack{x\to y\\ x\in \Omega_2}}\E_2^+(x,t)
\end{align*} 
for $t>t_0$. The fields $\H_1^+(y,t)$ and $\H_2^+(y,t)$ are defined similarly.
Next, set
\[
\G
=
\begin{cases}
\G^-(x,t)=\e^-(x,t)\E^-(x,t) & \text{for $x\in \Omega$ and $a<t<t_0$}\\
\G^+_1(x,t)=\e^+_1(x,t)\E^+_1(x,t) & \text{for $x\in \Omega_1$ and $t_0<t<b$}\\
\G^+_2(x,t)=\e^+_2(x,t)\E^+_2(x,t) & \text{for $x\in \Omega_2$ and $t_0<t<b$}.
\end{cases}
\]
Applying Proposition~\ref{prop:timederivativeforGfortwomedia} to the field $\G$ defined above gives
\begin{align}\label{eq:formula for derivative with respect to t of E}
\left\langle \dfrac{\partial \G}{\partial t},\phi \right\rangle
&=
\int_{\Omega_1} \(\G_1^+(x,t_0^+)-\G^-(x,t_0^-)\)\phi(x,t_0)\,dx
+\int_{\Omega_2} \(\G_2^+(x,t_0^+)-\G^-(x,t_0^-)\)\phi(x,t_0)\,dx \\
&\qquad +\int_{\Omega} \int_a^{t_0} \partial_t\G^-(x,t)\phi(x,t)\,dt\,dx
+\int_{\Omega_1}\int_{t_0}^b \partial_t\G_1^+(x,t)\phi(x,t)\,dt\,dx\notag\\
&\qquad +\int_{\Omega_2} \int_{t_0}^b \partial_t\G_2^+(x,t)\phi(x,t)\,dt\,dx\notag\\
&=\int_{\Omega_1} \(\e^+_1(x,t_0^+)\E_1^+(x,t_0^+)-\e^-(x,t_0^-)\E^-(x,t_0^-)\)\phi(x,t_0)\,dx\notag\\
&\qquad \qquad
+\int_{\Omega_2} \(\e^+_2(x,t_0^+)\E_2^+(x,t_0^+)-\e^-(x,t_0^-)\E^-(x,t_0^-)\)\phi(x,t_0)\,dx \notag\\
&\qquad +\int_{\Omega} \int_a^{t_0} \partial_t\(\e^-\E^-(x,t)\)\phi(x,t)\,dt\,dx 
+\int_{\Omega_1}\int_{t_0}^b \partial_t\(\e^+_1\E_1^+(x,t)\)\phi(x,t)\,dt\,dx\notag\\
&\qquad \qquad \qquad 
+\int_{\Omega_2} \int_{t_0}^b \partial_t\(\e^+_2\E_2^+(x,t)\)\phi(x,t)\,dt\,dx.\notag
\end{align}

On the other hand, applying Proposition~\ref{prop:formulaforcurlforgeneralmedia} with $\G=\H$ yields
\begin{align}
\langle \nabla \times \H, \phi \rangle
&=
-\int_{t_0}^b \int_\Gamma \(\(\H_2^+-\H_1^+\)\times \mathbf{n}\)\,\phi\,d\sigma\,dt
+\int_a^{t_0} \int_{\Omega}\phi\,\nabla \times \H^- \,dx\,dt\notag\\
&\qquad +\int_{t_0}^b\int_{\Omega_1}\phi\,\nabla \times\H_1^+ \,dx\,dt
+ \int_{t_0}^b\int_{\Omega_2}\phi\,\nabla \times\H_2^+ \,dx\,dt.\label{eq:formula for curl of H}
\end{align}
Assume that $\E$ and $\H$ satisfy the Maxwell equation  
\begin{equation}\label{eq:maxwell equation Faraday-latest}
\nabla \times \H=\dfrac{1}{c} \dfrac{\partial }{\partial t}\(\e \E\),
\end{equation}
in the sense of distributions in $\Omega\times (a,b)$.
Using the distributional formulas above, we deduce boundary conditions for $\E$ and $\H$.
More precisely, we prove the following spatio-temporal boundary conditions for the electric and magnetic fields.

\begin{theorem}\label{thm:temporal boundary conditions}
Assume that $\E$ and $\H$ satisfy Maxwell's equations \eqref{eq:maxwell equation Faraday-latest} and 
\eqref{eq:maxwell equation Ampere new} in the sense of distributions in $\Omega\times (a,b)$. 
Then the electric field satisfies
\begin{align}\label{first_new_temporal_bc}
\e^+_1(x,t_0^+)\E_1^+(x,t_0^+)-\e^-(x,t_0^-)\E^-(x,t_0^-)=0\qquad \text{for all } x\in \Omega_1,
\end{align}
\begin{align}\label{second_new_temporal_bc}
\e^+_2(x,t_0^+)\E_2^+(x,t_0^+)-\e^-(x,t_0^-)\E^-(x,t_0^-)=0\qquad \text{for all } x\in \Omega_2,
\end{align}
\begin{equation}\label{eq:space boundary condition for E}
\(\E_2^+(y,t)-\E_1^+(y,t)\)\times \mathbf{n}(y)=0 \quad \text{for each } y\in \Gamma \text{ and } t_0<t<b,
\end{equation}
and
\begin{align}\label{eq:bdry condition for E with dot product}
\(\e_2^+(y,t)\E_2^+(y,t)-\e_1^+(y,t)\E_1^+(y,t)\) \cdot \mathbf{n}(y) =0, \qquad \text{for each } y\in \Gamma \text{ and } t_0<t<b.
\end{align}

In addition, the magnetic field satisfies
\begin{align}\label{first_new_temporal_bc_forH}
\mu^+_1(x,t_0^+)\H_1^+(x,t_0^+)-\mu^-(x,t_0^-)\H^-(x,t_0^-)=0\qquad \text{for all } x\in \Omega_1,
\end{align}
\begin{align}\label{second_new_temporal_bc_forH}
\mu^+_2(x,t_0^+)\H_2^+(x,t_0^+)-\mu^-(x,t_0^-)\H^-(x,t_0^-)=0\qquad \text{for all } x\in \Omega_2,
\end{align}
\begin{equation}\label{eq:space boundary condition}
\(\H_2^+(y,t)-\H_1^+(y,t)\)\times \mathbf{n}(y)=0 \quad \text{for each } y\in \Gamma \text{ and } t_0<t<b,
\end{equation}
and
\begin{align}\label{eq:bdry condition for H with dot product}
\(\mu_2^+(y,t)\H_2^+(y,t)-\mu_1^+(y,t)\H_1^+(y,t)\) \cdot \mathbf{n}(y) =0, \qquad \text{for each } y\in \Gamma \text{ and } t_0<t<b.
\end{align}

\end{theorem}

\begin{proof}
{\bf Case 1.} Suppose $\phi$ has compact support in $\Omega_1\times (a,t_0)$.
From \eqref{eq:formula for curl of H}, we obtain
\[
\langle \nabla \times \H, \phi \rangle
=
\int_a^{t_0} \int_{\Omega_1}\phi\,\nabla \times \H^- \,dx\,dt.
\]
Similarly, from \eqref{eq:formula for derivative with respect to t of E},
\[
\left\langle \dfrac{\partial (\e\E)}{\partial t},\phi \right\rangle
=
\int_{\Omega_1} \int_a^{t_0} \partial_t\(\e^-\E^-(x,t)\)\phi(x,t)\,dt\,dx.
\]
Hence, from \eqref{eq:maxwell equation Faraday-latest}, we obtain 
\[
\nabla \times \H^-(x,t)=\dfrac{1}{c}\,\partial_t\(\e^-\E^-(x,t)\)\qquad \text{for all $x\in \Omega_1$ and $a<t<t_0$.}\]

{\bf Case 2.} Suppose $\phi$ has compact support in $\Omega_2\times (a,t_0)$.
From \eqref{eq:formula for curl of H}, we obtain
\[
\langle \nabla \times \H, \phi \rangle
=
\int_a^{t_0} \int_{\Omega_2}\phi\,\nabla \times \H^- \,dx\,dt.
\]
Similarly, from \eqref{eq:formula for derivative with respect to t of E},
\[
\left\langle \dfrac{\partial (\e\E)}{\partial t},\phi \right\rangle
=
\int_{\Omega_2} \int_a^{t_0} \partial_t\(\e^-\E^-(x,t)\)\phi(x,t)\,dt\,dx.
\]
Hence, from \eqref{eq:maxwell equation Faraday-latest}, we obtain 
\[
\nabla \times \H^-(x,t)=\dfrac{1}{c}\,\partial_t\(\e^-\E^-(x,t)\)\qquad \text{for all $x\in \Omega_2$ and $a<t<t_0$.}\]

{\bf Case 3.} Suppose $\phi$ has compact support in $\Omega_1\times (t_0,b)$.
From \eqref{eq:formula for curl of H}, we obtain
\[
\langle \nabla \times \H, \phi \rangle
=
\int_{t_0}^b \int_{\Omega_1}\phi\,\nabla \times \H^+_1 \,dx\,dt.
\]
Similarly, from \eqref{eq:formula for derivative with respect to t of E},
\[
\left\langle \dfrac{\partial (\e\E)}{\partial t},\phi \right\rangle
=
\int_{\Omega_1} \int_{t_0}^b \partial_t\(\e^+_1\E^+_1(x,t)\)\phi(x,t)\,dt\,dx.
\]
Hence, from \eqref{eq:maxwell equation Faraday-latest}, we obtain 
\[
\nabla \times \H^+_1(x,t)=\dfrac{1}{c}\,\partial_t\(\e^+_1\E^+_1(x,t)\)\qquad \text{for all $x\in \Omega_1$ and $t_0<t<b$.}\]

{\bf Case 4.} Suppose $\phi$ has compact support in $\Omega_2\times (t_0,b)$.
From \eqref{eq:formula for curl of H}, we obtain
\[
\langle \nabla \times \H, \phi \rangle
=
\int_{t_0}^b \int_{\Omega_2}\phi\,\nabla \times \H^+_2 \,dx\,dt.
\]
Similarly, from \eqref{eq:formula for derivative with respect to t of E},
\[
\left\langle \dfrac{\partial (\e\E)}{\partial t},\phi \right\rangle
=
\int_{\Omega_2} \int_{t_0}^b \partial_t\(\e^+_2\E^+_2(x,t)\)\phi(x,t)\,dt\,dx.
\]
Hence, from \eqref{eq:maxwell equation Faraday-latest}, we obtain 
\[
\nabla \times \H^+_2(x,t)=\dfrac{1}{c}\,\partial_t\(\e^+_2\E^+_2(x,t)\)\qquad \text{for all $x\in \Omega_2$ and $t_0<t<b$.}\]

Therefore, using \eqref{eq:maxwell equation Faraday-latest} in \eqref{eq:formula for curl of H} and \eqref{eq:formula for derivative with respect to t of E}, the corresponding double integrals are equal, and we obtain
\begin{align}\label{eq:identity for the boundary}
&\int_{\Omega_1} \(\e^+_1(x,t_0^+)\E_1^+(x,t_0^+)-\e^-(x,t_0^-)\E^-(x,t_0^-)\)\phi(x,t_0)\,dx\notag
+\int_{\Omega_2} \(\e^+_2(x,t_0^+)\E_2^+(x,t_0^+)-\e^-(x,t_0^-)\E^-(x,t_0^-)\)\phi(x,t_0)\,dx\\
&\qquad 
=
-c\,\int_{t_0}^b \int_\Gamma \(\(\H_2^+-\H_1^+\)\times {\bf{n}}\)\,\phi\,d\sigma\,dt
\end{align}
for each function $\phi\in C_0^1(\Omega\times (a,b))$.
Now, if $\phi\in C_0^1(\Omega_1\times (a,b))$, then $\phi=0$ on $\Gamma$. Thus the right-hand side of the last identity is zero, and we obtain 
\[
\int_{\Omega_1} \(\e^+_1(x,t_0^+)\E_1^+(x,t_0^+)-\e^-(x,t_0^-)\E^-(x,t_0^-)\)\phi(x,t_0)\,dx=0
\]
for each $\phi\in C_0^1(\Omega_1\times (a,b))$. Therefore, by the fundamental lemma of the calculus of variations, we obtain the first temporal boundary condition \eqref{first_new_temporal_bc}.
Similarly, if $\phi\in C_0^1(\Omega_2\times (a,b))$, we obtain the second temporal boundary condition \eqref{second_new_temporal_bc}.
Hence, inserting \eqref{first_new_temporal_bc} and \eqref{second_new_temporal_bc} into \eqref{eq:identity for the boundary}, we obtain 
\[
\int_{t_0}^b \int_\Gamma \(\(\H_2^+-\H_1^+\)\times {\bf{n}}\)\,\phi\,d\sigma\,dt =0
\]
for each function $\phi\in C_0^1(\Omega\times (a,b))$, which implies \eqref{eq:space boundary condition}.
Similarly, suppose that the magnetic field $\H$ has the form
\begin{equation*}
\H=
\begin{cases}
\H^- & \text{for $t<t_0$ and $x\in \Omega$}\\
\H^+_1 & \text{for $t>t_0$ and $x\in \Omega_1$}\\
\H^+_2 & \text{for $t>t_0$ and $x\in \Omega_2$}.
\end{cases}
\end{equation*}
Next, set
\[
\G
=
\begin{cases}
\G^-(x,t)=\mu^-(x,t)\H^-(x,t) & \text{for $x\in \Omega$ and $a<t<t_0$}\\
\G^+_1(x,t)=\mu^+_1(x,t)\H^+_1(x,t) & \text{for $x\in \Omega_1$ and $t_0<t<b$}\\
\G^+_2(x,t)=\mu^+_2(x,t)\H^+_2(x,t) & \text{for $x\in \Omega_2$ and $t_0<t<b$}.
\end{cases}
\]
Assume now that the following holds in $\Omega \times (a,b)$:
\begin{equation}\label{eq:maxwell equation Ampere new}
\nabla \times \E=-\dfrac{1}{c} \dfrac{\partial }{\partial t}\(\mu \H\)
\end{equation}  
in the sense of distributions. As above, this means that
\begin{align}
\nabla \times \E^-(x,t)
&=-\dfrac{1}{c} \dfrac{\partial }{\partial t}\(\mu^-(x,t) \H^-(x,t)\),
&& x\in \Omega,\quad a<t<t_0, \label{eq:faraday_minus}\\
\nabla \times \E_1^+(x,t)
&=-\dfrac{1}{c} \dfrac{\partial }{\partial t}\(\mu^+_1(x,t) \H_1^+(x,t)\),
&& x\in \Omega_1,\quad t_0<t<b, \label{eq:faraday_plus_Omega1}\\
\nabla \times \E_2^+(x,t)
&=-\dfrac{1}{c} \dfrac{\partial }{\partial t}\(\mu^+_2(x,t) \H_2^+(x,t)\),
&& x\in \Omega_2,\quad t_0<t<b. \label{eq:faraday_plus_Omega2}
\end{align}
Repeating the previous argument, we obtain \eqref{first_new_temporal_bc_forH} and \eqref{second_new_temporal_bc_forH}, the analogues of \eqref{first_new_temporal_bc} and \eqref{second_new_temporal_bc}, as well as \eqref{eq:space boundary condition for E}, the analogue of \eqref{eq:space boundary condition}.

Now we consider the source-free divergence conditions obtained from \eqref{divergence E zero}, \eqref{divergence B zero}, and the constitutive relations \eqref{eq:constitutive}:
\begin{align}
\nabla \cdot (\varepsilon \E)&=0, \label{eq:divergence_free_electric_displacement}\\
\nabla \cdot (\mu \H)&=0, \label{eq:divergence_free_magnetic_induction}
\end{align}
again understood in the sense of distributions.
Using Proposition~\ref{prop:divergence_formula} and repeating the same procedure, first with $\G_2^+=\e_2^+\E_2^+$ and $\G_1^+=\e_1^+\E_1^+$, we obtain
\begin{align*}
&\nabla \cdot (\varepsilon^- \E^-)=0 \qquad \text{for each $x\in\Omega$ and $a<t<t_0$}, \\
&\nabla \cdot (\varepsilon_1^+ \E_1^+)=0 \qquad \text{for each $x\in\Omega_1$ and $t_0<t<b$}, \\
&\nabla \cdot (\varepsilon_2^+ \E_2^+)=0 \qquad \text{for each $x\in\Omega_2$ and $t_0<t<b$}, 
\end{align*}
and ultimately the boundary condition \eqref{eq:bdry condition for E with dot product}.
Similarly, taking $\G_2^+=\mu_2^+\H_2^+$ and $\G_1^+=\mu_1^+\H_1^+$, we have
\begin{align*}
&\nabla \cdot (\mu^- \H^-)=0 \qquad \text{for each $x\in\Omega$ and $a<t<t_0$}, \\
&\nabla \cdot (\mu_1^+  \H_1^+)=0 \qquad \text{for each $x\in\Omega_1$ and $t_0<t<b$}, \\
&\nabla \cdot (\mu_2^+ \H_2^+)=0 \qquad \text{for each $x\in\Omega_2$ and $t_0<t<b$.}
\end{align*}
Then we arrive at the boundary condition \eqref{eq:bdry condition for H with dot product} for $\H$.
This completes the proof of Theorem \ref{thm:temporal boundary conditions}.
\end{proof}


\section{The generalized Snell's law for spatio-temporal media}\label{sec:GSL_spacetime}

In this section, we prove the generalized Snell's law for media with both temporal and spatial interfaces. 
We shall use the following exponential lemma, whose proof is given in \cite[Lemma~1]{CGES25temporal}.

\begin{lemma}\label{lm:exponentials}
Let $A_1, \cdots, A_N \in \C^n\setminus \{0\}$, and $\o_1, \cdots, \o_N \in \R$.
If 
\begin{equation}\label{eq:sum exponentials equal zero}
\sum_{j=1}^N A_j\,e^{i\,\o_j x} = 0 \qquad \forall \; x \in \R,
\end{equation}
 then $\o_1=\cdots =\o_N$.
 \end{lemma}

The conclusion of Lemma~\ref{lm:exponentials} remains valid if \eqref{eq:sum exponentials equal zero} is assumed only for every $x$ in a nonempty open interval $(a,b)$. Indeed, the left-hand side of \eqref{eq:sum exponentials equal zero} extends to an entire vector-valued function; if it vanishes on an open interval, then it vanishes identically, and Lemma~\ref{lm:exponentials} applies. We use this open-interval version below.

\subsection{The temporal interface}\label{sec:temporal_interface}
We assume that the domain $\Omega$ is decomposed as described at the beginning of Section~\ref{sec:general_formulas}.
An incident wave propagates through the material in $\Omega$ for $t<t_0$. For $t>t_0$, the material properties change according to \eqref{eps-mu-piecewise-def}. Our objective is therefore to determine quantitatively the resulting propagation of the wave after this change. We represent the fields as combinations of piecewise plane waves.

With $\e$ and $\mu$ defined as in \eqref{eps-mu-piecewise-def}, define the velocity by
\[
v(x,t)=\dfrac{1}{\sqrt{\e(x,t)\mu(x,t)}}, \qquad t\in (a,b)\setminus \{t_0\}, \; x\in\Omega.
\]
In this section, we assume that each component of the material parameters in \eqref{eps-mu-piecewise-def}, namely $\e^-$, $\mu^-$, $\e^{+}_{1,2}$, and $\mu_{1,2}^{+}$, is independent of $x$ but may still vary with $t$. Thus, we denote the corresponding piecewise values of $v$ by
\[
v(x,t) =
\begin{cases}
v_-(t) & \text{if } x\in \Omega,\ t<t_0, \\
v_1^+(t) & \text{if } x\in \Omega_1,\ t>t_0,\\
v_2^+(t) & \text{if } x\in \Omega_2,\ t>t_0.
\end{cases}
\]
We also define the one-sided limits
\begin{align*}
v_-(t_0) &=\lim_{t\to t_0^-} v_-(t),\qquad 
v^1_+(t_0)=\lim_{t\to t_0^+}v_1^+(t),\qquad 
v^2_+(t_0)=\lim_{t\to t_0^+}v_2^+(t).
\end{align*}

We make the following ansatz for the electric fields. Namely, we assume that the incident electric field has the form
\[
\E_i(x,t)={\mathbf{I}}\, e^{i\omega_1 \left( \frac{\k_i \cdot x}{v_-(t)}-t\right)}, \qquad t < t_0, \quad x \in \Omega,\quad \omega_1>0,
\]
where ${\mathbf{I}}$ is a nonzero vector. We assume that the transmitted field has the form
\[
\E_t(x,t) =
{\mathbf{T}}\, e^{i\omega_3 \left( \frac{\k_t^1\cdot x}{v_+^1(t)}-t\right)}
\qquad \text{for } x\in\Omega_1,\ t>t_0.
\]
Because of the temporal change at $t=t_0$, a reflected field is also generated in $\Omega_1$. We assume that it has the form
\[
\E_r(x,t) =
{\mathbf{R}}\, e^{i\omega_2 \left( \frac{\k_r^1\cdot x}{v_+^1(t)}-t\right)}
\qquad \text{for } x\in\Omega_1,\ t>t_0.
\]
Here $\k_i$ and $\k_\ell^1$, $\ell=r,t$, are unit vectors.


As the wave propagates into $\Omega_2$, the spatial interface $\Gamma$ generates additional transmitted and reflected fields. We model these fields in Section~\ref{sec:spatial_interface}.

Set $\E_- = \E_i$ and $\E_+=\E_t+\E_r$ to denote the fields on the two sides of the temporal interface. Define	
\[
\E=
\begin{cases}
\E_+ & \text{for $x\in \Omega_1$ and $t>t_0$}\\
\E_- & \text{for $x\in \Omega$ and $t<t_0$}
\end{cases},
\]
We assume that $\E$ satisfies the Maxwell system in $\Omega_1\times (a,b)$. Applying the temporal boundary condition from \cite[equation (15)]{CGES25temporal}, which is analogous to \eqref{first_new_temporal_bc}, yields
\begin{align}\label{eq:bdry condition in Omega 1}
\e_1^+(x,t_0^+)\left( {\mathbf{T}}\, e^{i\omega_3 \left( \frac{\k_t^1\cdot x}{v_+^1(t_0)}-t_0\right)}
+{\mathbf{R}}\, e^{i\omega_2 \left( \frac{\k_r^1\cdot x}{v_+^1(t_0)}-t_0\right)}\right) -
\e^-(x,t_0^-)
{\mathbf{I}}\, e^{i\omega_1 \left( \frac{\k_i \cdot x}{v_-(t_0)}-t_0\right)}=0
\end{align}
for all $x=(x_1,x_2,x_3) \in\Omega_1$. 
%
%
Since the coefficients $\e_1^+(x,t_0^+)$ and $\e^-(x,t_0^-)$ are independent of $x$,  
and the vectors ${\mathbf{T}}$, ${\mathbf{R}}$, and ${\mathbf{I}}$ are nonzero,
Lemma~\ref{lm:equal_exponents_from_vector_equation} applied to \eqref{eq:bdry condition in Omega 1} gives
\begin{align*}
\omega_3  \frac{\k_t^1}{v_+^1(t_0)}
&=
\omega_2  \frac{\k_r^1}{v_+^1(t_0)}
=
\omega_1  \frac{\k_i}{v_-(t_0)},
\end{align*}
that is,
\begin{align}\label{eq:transmitted and reflected vectors in the first medium}
\k_t^1
&=
\dfrac{\omega_1}{\omega_3}\dfrac{v_+^1(t_0)}{v_-(t_0)}\k_i,
\qquad \text{and }
\k_r^1
=
\dfrac{\omega_1}{\omega_2}\dfrac{v_+^1(t_0)}{v_-(t_0)}\k_i.
\end{align}

The elementary fact about vector-valued exponential sums used above is recorded in the following lemma.

\begin{lemma}\label{lm:equal_exponents_from_vector_equation}
Let $\Omega\subset\R^3$ be an open connected domain. Assume that $B_i, B_r, B_t\in \C^3$ are nonzero vectors, $\m_i,\m_r,\m_t\in\R^3$, and
\begin{equation}\label{eq:lemma_vector_exponential_equation}
B_t e^{i\m_t\cdot x}+B_r e^{i\m_r\cdot x}-B_i e^{i\m_i\cdot x}=0
\qquad\text{for all }x\in\Omega.
\end{equation}
Then
\[
\m_t=\m_r=\m_i.
\]
\end{lemma}

\begin{proof}
Write
\[
B_t=(\alpha_1,\beta_1,\gamma_1),\qquad
B_r=(\alpha_2,\beta_2,\gamma_2),\qquad
B_i=(\alpha_3,\beta_3,\gamma_3).
\]
Then \eqref{eq:lemma_vector_exponential_equation} is equivalent to the three scalar equations
\begin{align}
\alpha_1 e^{i\m_t\cdot x}+\alpha_2 e^{i\m_r\cdot x}-\alpha_3 e^{i\m_i\cdot x}&=0,\label{eq:scalar_alpha}\\
\beta_1 e^{i\m_t\cdot x}+\beta_2 e^{i\m_r\cdot x}-\beta_3 e^{i\m_i\cdot x}&=0,\label{eq:scalar_beta}\\
\gamma_1 e^{i\m_t\cdot x}+\gamma_2 e^{i\m_r\cdot x}-\gamma_3 e^{i\m_i\cdot x}&=0.\label{eq:scalar_gamma}
\end{align}

Each component of the left-hand side of \eqref{eq:lemma_vector_exponential_equation} is a real-analytic function on $\R^3$. Since \eqref{eq:lemma_vector_exponential_equation} holds on the nonempty open set $\Omega$, the identity theorem for real-analytic functions implies that it holds for every $x\in\R^3$. Consequently, the scalar equations \eqref{eq:scalar_alpha}--\eqref{eq:scalar_gamma} also hold for every $x\in\R^3$.

We may therefore restrict these equations to the coordinate axes. For $q=1,2,3$, set $x=s\mathbf e_q$, where $\mathbf e_q$ is the $q$th coordinate vector and $s\in\R$. For example, \eqref{eq:scalar_alpha} becomes
\[
\alpha_1 e^{i(\m_t)_q s}
+\alpha_2 e^{i(\m_r)_q s}
-\alpha_3 e^{i(\m_i)_q s}=0
\qquad\text{for all }s\in\R.
\]
After the zero coefficients are discarded, all remaining coefficients are nonzero. Lemma~\ref{lm:exponentials} therefore gives equality of the corresponding $q$th components. Repeating this for $q=1,2,3$ gives equality of the corresponding vectors in $\R^3$.

Next, note that a nontrivial scalar equation among
\eqref{eq:scalar_alpha}--\eqref{eq:scalar_gamma} cannot have only one nonzero coefficient. Indeed, if, for example,
\[
a\,e^{i\m_t\cdot x}=0
\qquad\text{for all }x\in\R^3,
\]
then necessarily $a=0$, which contradicts the assumption that the equation is nontrivial. Hence every nontrivial scalar equation has either exactly two nonzero coefficients or all three coefficients nonzero.

If one of the equations \eqref{eq:scalar_alpha}--\eqref{eq:scalar_gamma} has all three coefficients nonzero, then the preceding coordinate-axis argument yields
\[
\m_t=\m_r=\m_i,
\]
and we are done.

Assume, therefore, that no nontrivial scalar equation has three nonzero coefficients. Then every nontrivial scalar equation has exactly two nonzero coefficients. In that case, the same coordinate-axis argument shows that each nontrivial equation yields equality of the two corresponding exponent vectors. For instance:

\begin{itemize}
\item If \eqref{eq:scalar_alpha} has $\alpha_1,\alpha_2\neq 0$ and $\alpha_3=0$, then
\[
\m_t=\m_r;
\]

\item If \eqref{eq:scalar_beta} has $\beta_1,\beta_3\neq 0$ and $\beta_2=0$, then
\[
\m_t=\m_i;
\]

\item If \eqref{eq:scalar_gamma} has $\gamma_2,\gamma_3\neq 0$ and $\gamma_1=0$, then
\[
\m_r=\m_i.
\]
\end{itemize}

It is convenient to encode this information in a graph. Consider the graph with vertices
\[
\{t,r,i\},
\]
and draw an edge between two vertices whenever one of the scalar equations
\eqref{eq:scalar_alpha}--\eqref{eq:scalar_gamma} has exactly the corresponding two coefficients nonzero. By the previous paragraph, each such edge gives the equality of the corresponding exponents.

Since each of the vectors $B_t, B_r, B_i$ is nonzero, each of the vertices $t,r,i$ appears in at least one of the equations \eqref{eq:scalar_alpha}--\eqref{eq:scalar_gamma} with a nonzero coefficient. Therefore each vertex is incident to at least one edge of the graph. Since the graph has only three vertices, it is connected. Consequently, there is a path joining any two vertices, and chaining the equalities given by the edges, we obtain
\[
\m_t=\m_r=\m_i.
\]

This proves the lemma.
\end{proof}



This shows the relationship between the wave vectors in the first layer $\Omega_1$ and the incident vector $\k_i$. Since the wave vectors are unit vectors, taking absolute values in \eqref{eq:transmitted and reflected vectors in the first medium} gives the relation
\[
|\omega_3|=|\omega_2|=\omega_1 \,\dfrac{v_+^1(t_0)}{v_-(t_0)},
\]
since $\omega_1>0$.

\color{black}
\subsection{The spatial interface}\label{sec:spatial_interface}

We now model the transmitted and reflected fields generated when the wave created at the temporal interface strikes the transition surface $\Gamma$ at times $t>t_0$. Assume that, for $t>t_0$, the field incident on $\Gamma$ from $\Omega_1$ is
\[
\tilde{\E_i}=\E_\t+\E_r
=
{\mathbf T}\,e^{i\omega_3\left(\frac{\k_\t^1\cdot x}{v_+^1(t)}-t\right)}
+
{\mathbf R}\,e^{i\omega_2\left(\frac{\k_r^1\cdot x}{v_+^1(t)}-t\right)},
\qquad x\in\Omega_1,\ t>t_0,
\]
where $\E_\t,\E_r$ were defined in Section \ref{sec:temporal_interface}. Here $\k_\t^1$ and $\k_r^1$ are the two wave vectors in $\Omega_1$ produced at the temporal interface, and they are related to the original incident vector $\k_i$ through \eqref{eq:transmitted and reflected vectors in the first medium}.

When the wave $\tilde{\E_i}$ reaches $\Gamma$, it generates a reflected field $\tilde{\E_r}$ in $\Omega_1$ and a transmitted field $\tilde{\E_\t}$ in $\Omega_2$.
Since the frequencies $\omega_2$ and $\omega_3$ may be different, we allow both frequencies to appear in both fields. Thus we make the ansatz
\begin{equation}\label{eq:transmitted field into Omega 2}
\tilde{\E_\t}
=
{\mathbf T}_1\,e^{i\omega_3\left(\frac{\m_\t^1\cdot x}{v_+^2(t)}-t\right)}
+
{\mathbf T}_2\,e^{i\omega_2\left(\frac{\m_\t^2\cdot x}{v_+^2(t)}-t\right)},
\qquad x\in\Omega_2,\ t>t_0,
\end{equation}
and
\begin{equation}\label{eq:field reflected back from Gamma}
\tilde{\E_r}
=
{\mathbf R}_1\,e^{i\omega_3\left(\frac{\m_r^1\cdot x}{v_+^1(t)}-t\right)}
+
{\mathbf R}_2\,e^{i\omega_2\left(\frac{\m_r^2\cdot x}{v_+^1(t)}-t\right)},
\qquad x\in\Omega_1,\ t>t_0,
\end{equation}
for some unknown unit vectors $\m_\t^1,\m_\t^2,\m_r^1,\m_r^2$.

We now apply the boundary condition \eqref{eq:space boundary condition for E}, where
\[
\E_2^+=\tilde{\E}_{\t},
\qquad
\E_1^+=\tilde{\E}_i+\tilde{\E}_r.
\]
Thus
\[
\E_2^+
=
{\mathbf T}_1\,e^{i\omega_3\left(\frac{\m_\t^1\cdot x}{v_+^2(t)}-t\right)}
+
{\mathbf T}_2\,e^{i\omega_2\left(\frac{\m_\t^2\cdot x}{v_+^2(t)}-t\right)},
\qquad x\in\Omega_2,\ t>t_0,
\]
and
\[
\E_1^+
=
{\mathbf T}\,e^{i\omega_3\left(\frac{\k_\t^1\cdot x}{v_+^1(t)}-t\right)}
+
{\mathbf R}\,e^{i\omega_2\left(\frac{\k_r^1\cdot x}{v_+^1(t)}-t\right)}
+
{\mathbf R}_1\,e^{i\omega_3\left(\frac{\m_r^1\cdot x}{v_+^1(t)}-t\right)}
+
{\mathbf R}_2\,e^{i\omega_2\left(\frac{\m_r^2\cdot x}{v_+^1(t)}-t\right)},
\qquad x\in\Omega_1,\ t>t_0.
\]

Assume now that $\Gamma$ is a plane with unit normal
\[
\mathbf n=(0,0,1),
\]
so that $x\in\Gamma$ can be written as $x=(x_1,x_2,0)$. Then \eqref{eq:space boundary condition for E} yields
\begin{align*}
0
&=
(\E_2^+-\E_1^+)\times \mathbf n \\
&=
({\mathbf T}_1\times \mathbf n)\,e^{i\omega_3\left(\frac{\m_\t^1\cdot x}{v_+^2(t)}-t\right)}
+
({\mathbf T}_2\times \mathbf n)\,e^{i\omega_2\left(\frac{\m_\t^2\cdot x}{v_+^2(t)}-t\right)} \\
&\qquad
-
({\mathbf T}\times \mathbf n)\,e^{i\omega_3\left(\frac{\k_\t^1\cdot x}{v_+^1(t)}-t\right)}
-
({\mathbf R}\times \mathbf n)\,e^{i\omega_2\left(\frac{\k_r^1\cdot x}{v_+^1(t)}-t\right)} \\
&\qquad
-
({\mathbf R}_1\times \mathbf n)\,e^{i\omega_3\left(\frac{\m_r^1\cdot x}{v_+^1(t)}-t\right)}
-
({\mathbf R}_2\times \mathbf n)\,e^{i\omega_2\left(\frac{\m_r^2\cdot x}{v_+^1(t)}-t\right)},
\qquad x\in\Gamma,\ t>t_0.
\end{align*}

Since $x_3=0$ on $\Gamma$, this equation only contains the first two components of the phase vectors. In particular, it imposes no condition on the third components of
\[
\m_\t^1,\quad \m_\t^2,\quad \m_r^1,\quad \m_r^2.
\]

Write
\begin{align*}
&{\mathbf T}=(T_1,T_2,T_3),\qquad
{\mathbf R}=(R_1,R_2,R_3),\qquad
{\mathbf T}_1=(T_1^1,T_1^2,T_1^3),\qquad
{\mathbf T}_2=(T_2^1,T_2^2,T_2^3),\\
&{\mathbf R}_1=(R_1^1,R_1^2,R_1^3),\qquad
{\mathbf R}_2=(R_2^1,R_2^2,R_2^3).
\end{align*}
Then
\begin{align*}
&{\mathbf T}\times \mathbf n=(T_2,-T_1,0),\qquad
{\mathbf R}\times \mathbf n=(R_2,-R_1,0),\\
&{\mathbf T}_1\times \mathbf n=(T_1^2,-T_1^1,0),\qquad
{\mathbf T}_2\times \mathbf n=(T_2^2,-T_2^1,0),\\
&{\mathbf R}_1\times \mathbf n=(R_1^2,-R_1^1,0),\qquad
{\mathbf R}_2\times \mathbf n=(R_2^2,-R_2^1,0).
\end{align*}
Hence the vector identity above is equivalent to the two scalar equations
\begin{align}
\label{eq:spatial electric equation for 2nd components}
&T_1^2\,e^{i\omega_3\left(\frac{\m_\t^1\cdot x}{v_+^2(t)}-t\right)}
+
T_2^2\,e^{i\omega_2\left(\frac{\m_\t^2\cdot x}{v_+^2(t)}-t\right)}
-
T_2\,e^{i\omega_3\left(\frac{\k_\t^1\cdot x}{v_+^1(t)}-t\right)} 
-
R_2\,e^{i\omega_2\left(\frac{\k_r^1\cdot x}{v_+^1(t)}-t\right)}
-
R_1^2\,e^{i\omega_3\left(\frac{\m_r^1\cdot x}{v_+^1(t)}-t\right)}
-
R_2^2\,e^{i\omega_2\left(\frac{\m_r^2\cdot x}{v_+^1(t)}-t\right)}
=0
\end{align}
and
\begin{align}
\label{eq:spatial electric equation for 1st components}
&-T_1^1\,e^{i\omega_3\left(\frac{\m_\t^1\cdot x}{v_+^2(t)}-t\right)}
-
T_2^1\,e^{i\omega_2\left(\frac{\m_\t^2\cdot x}{v_+^2(t)}-t\right)}
+
T_1\,e^{i\omega_3\left(\frac{\k_\t^1\cdot x}{v_+^1(t)}-t\right)} 
+
R_1\,e^{i\omega_2\left(\frac{\k_r^1\cdot x}{v_+^1(t)}-t\right)}
+
R_1^1\,e^{i\omega_3\left(\frac{\m_r^1\cdot x}{v_+^1(t)}-t\right)}
+
R_2^1\,e^{i\omega_2\left(\frac{\m_r^2\cdot x}{v_+^1(t)}-t\right)}
=0
\end{align}
for all $x=(x_1,x_2,0)\in\Gamma$ and all $t>t_0$.

To simplify notation, define the tangential phase vectors
\begin{align*}
\lambda_1(t)&=\omega_3\frac{(\m_\t^1)_\parallel}{v_+^2(t)},&
\lambda_2(t)&=\omega_2\frac{(\m_\t^2)_\parallel}{v_+^2(t)},&
\lambda_3(t)&=\omega_3\frac{(\k_\t^1)_\parallel}{v_+^1(t)},\\
\lambda_4(t)&=\omega_2\frac{(\k_r^1)_\parallel}{v_+^1(t)},&
\lambda_5(t)&=\omega_3\frac{(\m_r^1)_\parallel}{v_+^1(t)},&
\lambda_6(t)&=\omega_2\frac{(\m_r^2)_\parallel}{v_+^1(t)},
\end{align*}
where $(\,\cdot\,)_\parallel$ denotes the projection onto the $(x_1,x_2)$-plane.

Also set
\[
\alpha=(\alpha_1,\dots,\alpha_6)
:=
\(T_1^2\,e^{-i\omega_3 t},T_2^2\,e^{-i\omega_2 t},-T_2\,e^{-i\omega_3 t},-R_2\,e^{-i\omega_2 t},-R_1^2\,e^{-i\omega_3 t},-R_2^2\,e^{-i\omega_2 t}\),
\]
and
\[
\beta=\(\beta_1,\dots,\beta_6\)
:=
\(-T_1^1\,e^{-i\omega_3 t},-T_2^1\,e^{-i\omega_2 t},T_1\,e^{-i\omega_3 t},R_1\,e^{-i\omega_2 t},R_1^1\,e^{-i\omega_3 t},R_2^1\,e^{-i\omega_2 t}\).
\]
Then \eqref{eq:spatial electric equation for 2nd components} and \eqref{eq:spatial electric equation for 1st components} can be written as
\begin{equation}\label{eq:alpha equation}
\sum_{j=1}^6 \alpha_j(t)\,e^{i\lambda_j(t)\cdot (x_1,x_2)}=0,
\end{equation}
and
\begin{equation}\label{eq:beta equation}
\sum_{j=1}^6 \beta_j(t)\,e^{i\lambda_j(t)\cdot (x_1,x_2)}=0,
\end{equation}
for all $(x_1,x_2)\in\R^2$ and all $t>t_0$.
\cristiancomment{Notice that here the $\lambda_j$'s depend on $t$ and the Lemma \ref{lm:exponentials} is applicable for each fixed $t$. Lemma \ref{lm:exponentials} is stated for $x\in \R$ but here in the last equations the exponents depend on $x_1,x_2$. However, letting $x_1=0$ we applying the Lemma in $x_2$ yields a conclusion for the second components $\lambda_j$ and letting $x_2=0$ and applying the Lemma in $x_1$ yields a conclusion for the first components of $\lambda_j$.    }

Assume now that none of the vectors
\[
{\mathbf T},\ {\mathbf R},\ {\mathbf T}_1,\ {\mathbf T}_2,\ {\mathbf R}_1,\ {\mathbf R}_2
\]
is parallel to $\mathbf n$. Then for each $j=1,\dots,6$ the pair $(\alpha_j,\beta_j)$ is not equal to $(0,0)$. Indeed, each such pair is precisely the first two components of the corresponding vector crossed with $\mathbf n$. However, one may still have $\alpha_j=0$ or $\beta_j=0$ separately, so the two equations must be used together.

We now make the additional assumption that
\[
T_1\neq 0,\qquad T_2\neq 0,\qquad R_1\neq 0,\qquad R_2\neq 0.
\]
To explain this assumption, notice that it could be that $\alpha_1,\alpha_2,\alpha_3,\alpha_4$ are all different from zero and $\alpha_5=\alpha_6=0$; and $\beta_1=\beta_2=\beta_3=\beta_4=0$ and $\beta_5,\beta_6$ are both not zero;
and so $(\alpha_1,\beta_1),\cdots ,(\alpha_6,\beta_6)$ are all different from zero. 
In that case we obtain that the first four exponents in \eqref{eq:alpha equation} are equal, and the last two exponents in \eqref{eq:beta equation} are equal, and therefore there wouldn't be any relationship between the $\k$'s and the $\m_r$'s.
We make the above assumption to avoid this situation. 

Equivalently, the amplitudes ${\mathbf T}$ and ${\mathbf R}$ have both first and second components nonzero. This guarantees that
\[
\alpha_3,\alpha_4,\beta_3,\beta_4
\]
are all nonzero, so the exponentials corresponding to $\k_\t^1$ and $\k_r^1$ appear in both scalar equations \eqref{eq:alpha equation} and \eqref{eq:beta equation}. These two terms will serve as a bridge between the supports of the $\alpha$-equation and the $\beta$-equation.

For each fixed $t>t_0$, define
\[
A=\{j:\alpha_j\neq 0\},\qquad B=\{j:\beta_j\neq 0\}.
\]
Then $A\cup B=\{1,\dots,6\}$, because $(\alpha_j,\beta_j)\neq (0,0)$ for every $j$, and moreover
\[
3,4\in A\cap B.
\]

Apply Lemma \ref{lm:exponentials} to \eqref{eq:alpha equation}, after discarding the zero coefficients. Since all the remaining coefficients are nonzero, we obtain
\[
\lambda_j(t)=\lambda_\ell(t)\qquad \text{for all }j,\ell\in A.
\]
Similarly, applying Lemma \ref{lm:exponentials} to \eqref{eq:beta equation}, we get
\[
\lambda_j(t)=\lambda_\ell(t)\qquad \text{for all }j,\ell\in B.
\]
Because $3\in A\cap B$, it follows that every $\lambda_j(t)$ is equal to $\lambda_3(t)$: if $j\in A$, then $\lambda_j(t)=\lambda_3(t)$ by the first equation; if $j\in B$, then $\lambda_j(t)=\lambda_3(t)$ by the second. Since $A\cup B=\{1,\dots,6\}$, we conclude that
\[
\lambda_1(t)=\lambda_2(t)=\lambda_3(t)=\lambda_4(t)=\lambda_5(t)=\lambda_6(t),
\qquad t>t_0.
\]

Thus all tangential exponents are equal, that is,
\[
\omega_3\frac{(\m_\t^1)_\parallel}{v_+^2(t)}
=
\omega_2\frac{(\m_\t^2)_\parallel}{v_+^2(t)}
=
\omega_3\frac{(\k_\t^1)_\parallel}{v_+^1(t)}
=
\omega_2\frac{(\k_r^1)_\parallel}{v_+^1(t)}
=
\omega_3\frac{(\m_r^1)_\parallel}{v_+^1(t)}
=
\omega_2\frac{(\m_r^2)_\parallel}{v_+^1(t)}.
\]
Equivalently,
\[
(\k_\t^1)_\parallel
=
\frac{\omega_2}{\omega_3}(\k_r^1)_\parallel
=
(\m_r^1)_\parallel
=
\frac{\omega_2}{\omega_3}(\m_r^2)_\parallel,
\]
and
\[
(\m_\t^1)_\parallel
=
\frac{\omega_2}{\omega_3}(\m_\t^2)_\parallel,
\qquad
(\m_\t^2)_\parallel
=
\frac{\omega_3}{\omega_2}\frac{v_+^2(t_0)}{v_+^1(t_0)}(\k_\t^1)_\parallel.
\]
Using \eqref{eq:transmitted and reflected vectors in the first medium}, we obtain
\begin{equation}\label{eq:system showing GSL}
\left\{
\begin{aligned}
(\m_r^1)_\parallel
&=
(\k_\t^1)_\parallel
=
\frac{\omega_1}{\omega_3}\frac{v_+^1(t_0)}{v_-(t_0)}(\k_i)_\parallel,\\[0.4em]
(\m_r^2)_\parallel
&=
(\k_r^1)_\parallel
=
\frac{\omega_1}{\omega_2}\frac{v_+^1(t_0)}{v_-(t_0)}(\k_i)_\parallel,\\[0.4em]
(\m_\t^1)_\parallel
&=
\frac{\omega_1}{\omega_3}\frac{v_+^2(t_0)}{v_-(t_0)}(\k_i)_\parallel,\\[0.4em]
(\m_\t^2)_\parallel
&=
\frac{\omega_1}{\omega_2}\frac{v_+^2(t_0)}{v_-(t_0)}(\k_i)_\parallel.
\end{aligned}
\right.
\end{equation}
These identities determine the first two components of the four unknown vectors and constitute the generalized Snell law for the tangential components.

The tangential phase relations above also simplify the electric boundary equations \eqref{eq:spatial electric equation for 2nd components} and \eqref{eq:spatial electric equation for 1st components}. For each fixed $t>t_0$, the equality
\[
\lambda_1(t)=\cdots=\lambda_6(t)=:\lambda(t)
\]
shows that every term in these equations contains the same nonzero spatial factor
\[
e^{i\lambda(t)\cdot(x_1,x_2)}.
\]
Dividing by this common factor removes the spatial dependence and leaves
\begin{align*}
&T_1^2\,e^{-i\omega_3t}
+
T_2^2\,e^{-i\omega_2t}
-
T_2\,e^{-i\omega_3t} 
-
R_2\,e^{-i\omega_2t}
-
R_1^2\,e^{-i\omega_3t}
-
R_2^2\,e^{-i\omega_2t}
=0
\end{align*}
and
\begin{align*}
&-T_1^1\,e^{-i\omega_3t}
-
T_2^1\,e^{-i\omega_2t}
+
T_1\,e^{-i\omega_3t} 
+
R_1\,e^{-i\omega_2t}
+
R_1^1\,e^{-i\omega_3t}
+
R_2^1\,e^{-i\omega_2t}
=0
\end{align*}
valid for all $t>t_0$.
Grouping the terms with the same temporal frequency gives
\begin{align}
\label{eq:exponential equation for 2nd components in t}
\(T_1^2-T_2-R_1^2\)\,e^{-i\omega_3t}
+
\(T_2^2-R_2-R_2^2\)\,e^{-i\omega_2t}
=0
\end{align}
and
\begin{align}
\label{eq:exponential equation for 1st components in t}
\(T_1^1-T_1-R_1^1\)\,e^{-i\omega_3t}
+
\(T_2^1-R_1-R_2^1\)\,e^{-i\omega_2t}
=0
\end{align}
for all $t>t_0$. 
Each of \eqref{eq:exponential equation for 2nd components in t} and \eqref{eq:exponential equation for 1st components in t} has the form
\[
A e^{-i\omega_3t}+B e^{-i\omega_2t}=0.
\]
If one coefficient vanishes, then the other must also vanish. If both coefficients are nonzero, Lemma~\ref{lm:exponentials}, applied in the variable $t$, implies that $\omega_2=\omega_3$. Thus, when $\omega_2\neq\omega_3$, both coefficients in each of the two equations must vanish; conversely, if either equation has two nonzero coefficients, then necessarily $\omega_2=\omega_3$.

To recover the third components, we use that the vectors $\m_\t^1,\m_\t^2,\m_r^1,\m_r^2$ are unit vectors. For example,
\[
(m_\t^1)_3
=
\pm\sqrt{
1-
\left(\frac{\omega_1}{\omega_3}\frac{v_+^2(t_0)}{v_-(t_0)}(k_i)_1\right)^2
-
\left(\frac{\omega_1}{\omega_3}\frac{v_+^2(t_0)}{v_-(t_0)}(k_i)_2\right)^2},
\]
and
\[
(m_\t^2)_3
=
\pm\sqrt{
1-
\left(\frac{\omega_1}{\omega_2}\frac{v_+^2(t_0)}{v_-(t_0)}(k_i)_1\right)^2
-
\left(\frac{\omega_1}{\omega_2}\frac{v_+^2(t_0)}{v_-(t_0)}(k_i)_2\right)^2}.
\]
The sign is chosen according to the direction of propagation of the corresponding wave. Analogous formulas hold for $(m_r^1)_3$ and $(m_r^2)_3$.

Therefore, \eqref{eq:system showing GSL}, together with the unit-length condition for the normal components, gives the generalized Snell law at the spatial interface.

\begin{theorem}[Generalized Snell law at the spatial interface]\label{thm:GSL_spatial_interface}
Let $\Gamma=\{x_3=0\}$ be a planar interface separating the media $\Omega_1$ and $\Omega_2$, with unit normal
$
\mathbf n=(0,0,1).
$
Let $\tilde{\E_i}$ be the field incident on $\Gamma$ defined at the beginning of Section~\ref{sec:spatial_interface}, and let $\tilde{\E_\t}$ and $\tilde{\E_r}$ be the transmitted and reflected fields defined in \eqref{eq:transmitted field into Omega 2} and \eqref{eq:field reflected back from Gamma}, respectively. The incident wave vectors $\k_\t^1$ and $\k_r^1$ are given by \eqref{eq:transmitted and reflected vectors in the first medium}, and $\m_\t^1,\m_\t^2,\m_r^1,\m_r^2$ are unit vectors.

Assume moreover that none of the vectors
\[
{\mathbf T},\ {\mathbf R},\ {\mathbf T}_1,\ {\mathbf T}_2,\ {\mathbf R}_1,\ {\mathbf R}_2
\]
is parallel to $\mathbf n$, and that the incident amplitudes satisfy
\[
T_1\neq 0,\qquad T_2\neq 0,\qquad R_1\neq 0,\qquad R_2\neq 0.
\]

Then the tangential components of the wave vectors satisfy the system \eqref{eq:system showing GSL}, where $(\,\cdot\,)_\parallel$ denotes the projection onto the tangent plane of $\Gamma$.

Equivalently,
\[
\omega_3\frac{(\m_\t^1)_\parallel}{v_+^2(t_0)}
=
\omega_2\frac{(\m_\t^2)_\parallel}{v_+^2(t_0)}
=
\omega_3\frac{(\k_\t^1)_\parallel}{v_+^1(t_0)}
=
\omega_2\frac{(\k_r^1)_\parallel}{v_+^1(t_0)}
=
\omega_3\frac{(\m_r^1)_\parallel}{v_+^1(t_0)}
=
\omega_2\frac{(\m_r^2)_\parallel}{v_+^1(t_0)}.
\]

The normal components are then determined by the unit-length condition:
\[
(m_\t^1)_3
=
\pm\sqrt{
1-\left|(\m_\t^1)_\parallel\right|^2},
\qquad
(m_\t^2)_3
=
\pm\sqrt{
1-\left|(\m_\t^2)_\parallel\right|^2},
\]
\[
(m_r^1)_3
=
\pm\sqrt{
1-\left|(\m_r^1)_\parallel\right|^2},
\qquad
(m_r^2)_3
=
\pm\sqrt{
1-\left|(\m_r^2)_\parallel\right|^2},
\]
where in each case the sign is chosen according to the physical direction of propagation. Therefore \eqref{eq:system showing GSL}, together with these formulas for the normal components, constitutes the generalized Snell law at the spatial interface $\Gamma$.
\end{theorem}

\color{black}
{\color{black}{
\begin{remark}
Note that in the last two equations of (\ref{eq:system showing GSL}), one cannot immediately apply absolute value to obtain a relationship between the frequencies $\omega_2, \omega_3$ and velocities $v_+^2(t_0), v_-(t_0)$ since these equations hold only for the first two components of the vectors.
\end{remark}

\setcounter{equation}{0}
\section{Calculation of the amplitudes of the fields}\label{sec:amplitude_calculation}

In this section,  we the phase information from Sections \ref{sec:temporal_interface} and \ref{sec:spatial_interface} into equations for the unknown field amplitudes. The organization is as follows.
\begin{enumerate}[label=\textup{(\roman*)}]
\item In Subsection~\ref{subsec:amplitude_stages}, we list the four wave-generation stages and fix the notation for the six unknown amplitudes.
\item In Subsection~\ref{subsec:electric_amplitude_equations} we derive the amplitude equations from the electric boundary conditions.
\item In Remark~\ref{rem:electric_equations_do_not_close} we explain why magnetic boundary conditions are needed; Subsection~\ref{subsec:magnetic_amplitude_equations} is then devoted to computing the corresponding magnetic boundary conditions.
\item In Subsections~\ref{subsec:full_amplitude_system} and \ref{subsec:amplitude_coordinate_system} we assemble the full linear system, state the solvability conditions, and write this system in coordinate form.
\item In Subsection~\ref{subsec:amplitude_example} we provide an explicit example at oblique-incidence.
\end{enumerate}

\subsection{Roadmap and wave-generation stages}\label{subsec:amplitude_stages}

The calculation proceeds in the following stages, which we outline here for ease of the reader:
\begin{enumerate}
\item[\bf Stage 1] The prescribed field is the incident field
\[ 
\E_i(x,t)={\mathbf{I}}\, e^{i\omega_1 \left( \frac{\k_i \cdot x}{v_-(t)}-t\right)}, \qquad t < t_0, \qquad x \in \Omega,\quad \omega_1>0,
\]
where all parameters involved in $\E_i$ are known. 
\item[\bf Stage 2] After the material changes at $t=t_0$, this field generates two fields:
a transmitted field
\[ 
\E_t(x,t) = {\mathbf{T}}\, e^{i\omega_3 \left( \frac{\k_t^1\cdot x}{v_+^1(t)}-t\right)} \quad \text{ for } \; x\in\Omega_1, \; t>t_0,	
\]
and a reflected field
\[ \E_r(x,t) = {\mathbf{R}}\, e^{i\omega_2 \left( \frac{\k_r^1\cdot x}{v_+^1(t)}-t\right)} \quad \text{ for } \; x\in\Omega_1, \; t>t_0,	
\]
Here all parameters involved in $\E_t$ and $\E_r$ are unknown, except for the phase vectors $\k_r^1$ and $\k_t^1$, which are calculated in \eqref{eq:transmitted and reflected vectors in the first medium}, and the velocity $v_+^1(t)$. 	
\item[\bf Stage 3] Additional fields appear due to the presence of the spatial interface surface $\Gamma$. The incident field on $\Gamma$, coming from Stage 2 and travelling through $\Omega_1$, is
$\tilde \E_i=\E_t+\E_r$, defined for $x\in \Omega_1$ and $t>t_0$. This field generates a reflected field, denoted by $\tilde \E_r$, which is also defined in $\Omega_1$ for $t>t_0$.
We postulate that $\tilde \E_r$ has the form given in \eqref{eq:field reflected back from Gamma}.
At this point the amplitudes and phase vectors of these fields are unknown.
\item[\bf Stage 4] In the final stage, the fields in $\Omega_1$ strike $\Gamma$ and pass into $\Omega_2$. From Stage 3, the resulting incident field on $\Gamma$ is
$\tilde \E_i+\tilde \E_r$, which yields a transmitted field in $\Omega_2$ for $t>t_0$. This transmitted field is denoted by $\tilde \E_t$ and is postulated to have the form \eqref{eq:transmitted field into Omega 2}.
\end{enumerate}

Figure~\ref{fig:four-stage-fields} summarizes how the fields are generated, combined, and scattered during the four stages.

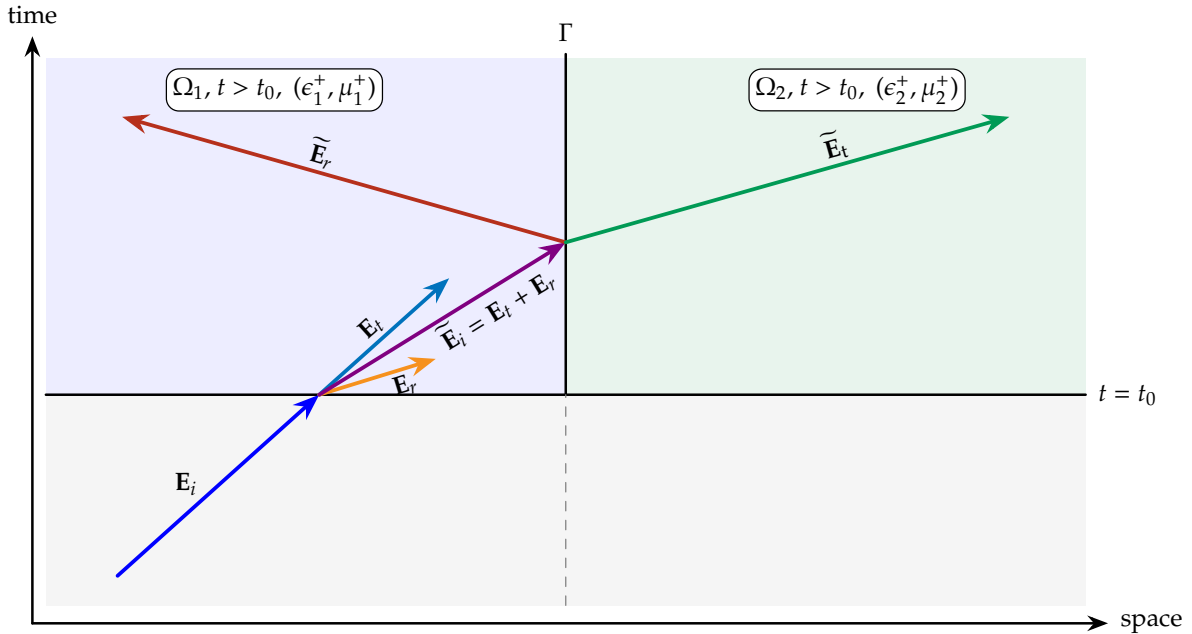
\begin{figure}[!htbp]
\centering
\resizebox{0.96\linewidth}{!}{%
\begin{tikzpicture}[>=Stealth, line cap=round, line join=round, every node/.style={font=\scriptsize,text=black}]
\fill[gray!8] (-5.8,-2.35) rectangle (5.8,0);
\fill[blue!7] (-5.8,0) rectangle (0,3.75);
\fill[ForestGreen!8] (0,0) rectangle (5.8,3.75);

\draw[->, thick] (-5.95,-2.55) -- (6.05,-2.55) node[right] {space};
\draw[->, thick] (-5.95,-2.55) -- (-5.95,4.0) node[above] {time};
\draw[thick] (-5.8,0) -- (5.8,0) node[right] {$t=t_0$};
\draw[thick] (0,0) -- (0,3.8) node[above] {$\Gamma$};
\draw[dashed, gray] (0,-2.35) -- (0,0);

\node[draw, rounded corners, fill=white, align=center, inner sep=2pt] at (-3.25,3.40)
{$\Omega_1$, $t>t_0,\; (\epsilon_1^+,\mu_1^+)$};
\node[draw, rounded corners, fill=white, align=center, inner sep=2pt] at (3.25,3.40)
{$\Omega_2$, $t>t_0,\; (\epsilon_2^+,\mu_2^+)$};


\coordinate (temporalpoint) at (-2.75,0);
\coordinate (transmittedpoint) at (-1.30,1.30);
\coordinate (reflectedpoint) at (-1.45,0.40);
\coordinate (gammapoint) at (0,1.70);

\draw[->, very thick, blue] (-5.0,-2.02) -- (temporalpoint)
node[pos=0.42, above left, fill=none, inner sep=1pt] {$\E_i$};

\draw[->, very thick, RoyalBlue]
  (temporalpoint) -- (transmittedpoint)
  node[pos=0.48, above, sloped, fill=none, inner sep=1pt]
  {$\E_t$};
\draw[->, very thick, BurntOrange]
  (temporalpoint) -- (reflectedpoint)
  node[pos=0.70, below, sloped, fill=none, inner sep=1pt]
  {$\E_r$};


\draw[->, very thick, violet] (temporalpoint) -- (gammapoint)
  node[pos=0.68, below, sloped, text=black, fill=none, inner sep=1pt]
  {$\widetilde{\E}_i=\E_t+\E_r$};

\draw[->, very thick, BrickRed] (gammapoint) -- (-4.95,3.10)
node[pos=0.56, above, sloped, fill=none, inner sep=1pt]
{$\widetilde{\E}_r$};

\draw[->, very thick, ForestGreen] (gammapoint) -- (4.95,3.10)
node[pos=0.62, above, sloped, fill=none, inner sep=1pt]
{$\widetilde{\E}_{\t}$};

\end{tikzpicture}%
}
\caption{Schematic of the four-stage space-time slab construction. The temporal jump at $t=t_0$ produces the fields $\E_t$ and $\E_r$ in $\Omega_1$. The dashed translated sides complete the parallelogram, whose purple diagonal is the sum $\widetilde{\E}_i=\E_t+\E_r$ incident on the spatial interface $\Gamma$. The reflected field $\widetilde{\E}_r$ in $\Omega_1$ has branches $(\omega_3,\m_r^1)$ and $(\omega_2,\m_r^2)$, while the transmitted field $\widetilde{\E}_{\t}$ in $\Omega_2$ has branches $(\omega_3,\m_\t^1)$ and $(\omega_2,\m_\t^2)$.}
\label{fig:four-stage-fields}
\end{figure}
\FloatBarrier

\subsection{Equations for the electric field amplitudes}\label{subsec:electric_amplitude_equations}

The phase vectors in the fields above are fixed by Sections \ref{sec:temporal_interface} and \ref{sec:spatial_interface}. We now determine the six remaining amplitudes
\[
{\mathbf T},\ {\mathbf R},\ {\mathbf R}_1,\ {\mathbf R}_2,\ {\mathbf T}_1,\ {\mathbf T}_2
\]
from the boundary conditions proved in Theorem \ref{thm:temporal boundary conditions}.
Recall \eqref{eq:bdry condition in Omega 1}:
\begin{align*}\label{eq:bdry condition in Omega 1}
\e_1^+(x,t_0^+)\left( {\mathbf{T}}\, e^{i\omega_3 \left( \frac{\k_t^1\cdot x}{v_+^1(t_0)}-t_0\right)}
+{\mathbf{R}}\, e^{i\omega_2 \left( \frac{\k_r^1\cdot x}{v_+^1(t_0)}-t_0\right)}\right) -
\e^-(x,t_0^-)
{\mathbf{I}}\, e^{i\omega_1 \left( \frac{\k_i \cdot x}{v_-(t_0)}-t_0\right)}=0
\end{align*}
for all $x=(x_1, x_2, x_3) \in \Omega_1$.
Using \eqref{eq:transmitted and reflected vectors in the first medium}, namely,
\begin{align*}
\k_t^1
&=
\dfrac{\omega_1}{\omega_3}  \frac{v_+^1(t_0)}{v_-(t_0)}\k_i,
\qquad \text{and }
\k_r^1
=
\dfrac{\omega_1}{\omega_2}  \frac{v_+^1(t_0)}{v_-(t_0)}\k_i,
\end{align*}
in \eqref{eq:transmitted and reflected vectors in the first medium}, the spatial exponential factors cancel. We therefore obtain the first equation for the amplitudes $\mathbf{T}$, $\mathbf{R}$, and $\mathbf{I}$:
\begin{align*}
\e_1^+(x,t_0^+)\left( {\mathbf{T}}\, e^{-i\omega_3 t_0}
+{\mathbf{R}}\, e^{-i\omega_2 t_0}\right) -
\e^-(x,t_0^-)
{\mathbf{I}}\, e^{-i\omega_1 t_0}=0
\end{align*}
for all $x\in \Omega_1$, where the unknowns are $\mathbf T$ and $\mathbf R$.

We next use the boundary condition \eqref{second_new_temporal_bc}, which we recall as
\begin{align*}
\e^+_2(x,t_0^+)\E_2^+(x,t_0^+)-\e^-(x,t_0^-)\E^-(x,t_0^-)=0\qquad \forall x\in \Omega_2,
\end{align*}
where 
\[
\E_2^+
=
{\mathbf T}_1\,e^{i\omega_3\left(\frac{\m_\t^1\cdot x}{v_+^2(t)}-t\right)}
+
{\mathbf T}_2\,e^{i\omega_2\left(\frac{\m_\t^2\cdot x}{v_+^2(t)}-t\right)},
\qquad x\in\Omega_2,\ t>t_0,
\]
and
\[
\E^-={\mathbf{I}}\, e^{i\omega_1 \left( \frac{\k_i \cdot x}{v_-(t_0)}-t_0\right)}.
\]
Thus we obtain
\[
\e^+_2(x,t_0^+)\({\mathbf T}_1\,e^{i\omega_3\left(\frac{\m_\t^1\cdot x}{v_+^2(t_0)}-t_0\right)}
+
{\mathbf T}_2\,e^{i\omega_2\left(\frac{\m_\t^2\cdot x}{v_+^2(t_0)}-t_0\right)}\)-\e^-(x,t_0^-){\mathbf{I}}\, e^{i\omega_1 \left( \frac{\k_i \cdot x}{v_-(t_0)}-t_0\right)}=0,
\]
for all $x\in \Omega_2$.
Proceeding as before, we obtain the following relation among the coefficients in the exponents:
\begin{align*}
\omega_3  \frac{\m_\t^1}{v_+^2(t_0)}
&=
\omega_2  \frac{\m_\t^2}{v_+^2(t_0)}
=
\omega_1  \frac{\k_i }{v_-(t_0)}.
\end{align*}
Equivalently,
\begin{align}\label{eq:transmitted and incident vectors in the second medium}
\m_\t^1
&=
\dfrac{\omega_1}{\omega_3}  \frac{v_+^2(t_0)}{v_-(t_0)}\k_i,
\qquad \text{and}\qquad
\m_\t^2
=
\dfrac{\omega_1}{\omega_2}  \frac{v_+^2(t_0)}{v_-(t_0)}\k_i.
\end{align}
Hence, using these relations, we can simplify the exponentials and obtain the equation
\[
\e^+_2(x,t_0^+)\({\mathbf T}_1\,e^{-i\omega_3 t_0}
+
{\mathbf T}_2\,e^{-i\omega_2 t_0}\)-\e^-(x,t_0^-){\mathbf{I}}\, e^{-i\omega_1 t_0}=0,
\]
for all $x\in \Omega_2$, where the unknowns are ${\mathbf T}_1$ and ${\mathbf T}_2$.

We now apply the boundary condition \eqref{eq:space boundary condition for E} with
\[
\E_2^+=\tilde{\E_\t},
\qquad
\E_1^+=\tilde{\E_i}+\tilde{\E_r}.
\]
That is, with
\[
\E_2^+
=
{\mathbf T}_1\,e^{i\omega_3\left(\frac{\m_\t^1\cdot x}{v_+^2(t)}-t\right)}
+
{\mathbf T}_2\,e^{i\omega_2\left(\frac{\m_\t^2\cdot x}{v_+^2(t)}-t\right)},
\qquad x\in\Omega_2,\ t>t_0,
\]
and
\[
\E_1^+
=
{\mathbf T}\,e^{i\omega_3\left(\frac{\k_\t^1\cdot x}{v_+^1(t)}-t\right)}
+
{\mathbf R}\,e^{i\omega_2\left(\frac{\k_r^1\cdot x}{v_+^1(t)}-t\right)}
+
{\mathbf R}_1\,e^{i\omega_3\left(\frac{\m_r^1\cdot x}{v_+^1(t)}-t\right)}
+
{\mathbf R}_2\,e^{i\omega_2\left(\frac{\m_r^2\cdot x}{v_+^1(t)}-t\right)},
\qquad x\in\Omega_1,\ t>t_0.
\]
Using \eqref{eq:transmitted and incident vectors in the second medium}, we obtain
\begin{align*}
\E_2^+
&=
{\mathbf T}_1\,e^{i\omega_3\left(\frac{\(\tfrac{\omega_1}{\omega_3}  \frac{v_+^2(t_0)}{v_-(t_0)}\k_i\)\cdot x}{v_+^2(t)}-t\right)}
+
{\mathbf T}_2\,e^{i\omega_2\left(\frac{\(\tfrac{\omega_1}{\omega_2}  \frac{v_+^2(t_0)}{v_-(t_0)}\k_i\)\cdot x}{v_+^2(t)}-t\right)}\\
&=
{\mathbf T}_1\,
e^{i\omega_1\left(\frac{\(  \frac{v_+^2(t_0)}{v_-(t_0)}\k_i\)\cdot x}{v_+^2(t)}\right)}\,e^{-i\,\omega_3\,t}
+
{\mathbf T}_2\,
e^{i\omega_1\left(\frac{ \( \frac{v_+^2(t_0)}{v_-(t_0)}\k_i\)\cdot x}{v_+^2(t)}\right)}\,e^{-i\,\omega_2\,t}\\
&=
\({\mathbf T}_1\,
e^{-i\,\omega_3\,t}
+
{\mathbf T}_2\,
e^{-i\,\omega_2\,t}\)\,e^{i\omega_1 \frac{v_+^2(t_0)}{v_-(t_0)v_+^2(t)}\k_i\cdot x} .
\end{align*}
By the preceding calculations, the only unknown phase vectors at this point are $\m_r^1$ and $\m_r^2$.

Now, as before, we assume that $\Gamma$ is a plane with unit normal
\[
\mathbf{n}=(0,0,1),
\]
so every $x\in\Gamma$ can be written as $x=(x_1,x_2,0)$. Then \eqref{eq:space boundary condition for E} yields
\begin{align}\label{vector-identity}
\begin{split}
0
&=
\(\E_2^+-\E_1^+\)\times \mathbf{n} \\
&=
({\mathbf T}_1\times \mathbf{n})\,e^{i\omega_3\left(\frac{\m_\t^1\cdot x}{v_+^2(t)}-t\right)}
+
({\mathbf T}_2\times \mathbf{n})\,e^{i\omega_2\left(\frac{\m_\t^2\cdot x}{v_+^2(t)}-t\right)} \\
&\qquad
-
({\mathbf T}\times \mathbf{n})\,e^{i\omega_3\left(\frac{\k_\t^1\cdot x}{v_+^1(t)}-t\right)}
-
({\mathbf R}\times \mathbf{n})\,e^{i\omega_2\left(\frac{\k_r^1\cdot x}{v_+^1(t)}-t\right)} \\
&\qquad
-
({\mathbf R}_1\times \mathbf{n})\,e^{i\omega_3\left(\frac{\m_r^1\cdot x}{v_+^1(t)}-t\right)}
-
({\mathbf R}_2\times \mathbf{n})\,e^{i\omega_2\left(\frac{\m_r^2\cdot x}{v_+^1(t)}-t\right)},
\qquad x\in\Gamma,\ t>t_0.
\end{split}
\end{align}

Since $x_3=0$ on $\Gamma$, this equation involves only the first two components of the phase vectors. In particular, it imposes no condition on the third components of the unknown vectors $\m_r^1$ and $\m_r^2$.

We write the amplitudes as
\begin{align*}
&{\mathbf T}=(T_1,T_2,T_3),\qquad
{\mathbf R}=(R_1,R_2,R_3),\qquad
{\mathbf T}_1=(T_1^1,T_1^2,T_1^3),\qquad
{\mathbf T}_2=(T_2^1,T_2^2,T_2^3),\\
&{\mathbf R}_1=(R_1^1,R_1^2,R_1^3),\qquad
{\mathbf R}_2=(R_2^1,R_2^2,R_2^3).
\end{align*}
As before, the vector identity (\ref{vector-identity}) is equivalent to the two scalar equations
\begin{align}
\label{eq:exponential equation for 2nd components}
&T_1^2\,e^{i\omega_3\left(\frac{\m_\t^1\cdot x}{v_+^2(t)}-t\right)}
+
T_2^2\,e^{i\omega_2\left(\frac{\m_\t^2\cdot x}{v_+^2(t)}-t\right)}
-
T_2\,e^{i\omega_3\left(\frac{\k_\t^1\cdot x}{v_+^1(t)}-t\right)} 
-
R_2\,e^{i\omega_2\left(\frac{\k_r^1\cdot x}{v_+^1(t)}-t\right)}
-
R_1^2\,e^{i\omega_3\left(\frac{\m_r^1\cdot x}{v_+^1(t)}-t\right)}
-
R_2^2\,e^{i\omega_2\left(\frac{\m_r^2\cdot x}{v_+^1(t)}-t\right)}
=0
\end{align}
and
\begin{align}
\label{eq:exponential equation for 1st components}
&-T_1^1\,e^{i\omega_3\left(\frac{\m_\t^1\cdot x}{v_+^2(t)}-t\right)}
-
T_2^1\,e^{i\omega_2\left(\frac{\m_\t^2\cdot x}{v_+^2(t)}-t\right)}
+
T_1\,e^{i\omega_3\left(\frac{\k_\t^1\cdot x}{v_+^1(t)}-t\right)} 
+
R_1\,e^{i\omega_2\left(\frac{\k_r^1\cdot x}{v_+^1(t)}-t\right)}
+
R_1^1\,e^{i\omega_3\left(\frac{\m_r^1\cdot x}{v_+^1(t)}-t\right)}
+
R_2^1\,e^{i\omega_2\left(\frac{\m_r^2\cdot x}{v_+^1(t)}-t\right)}
=0
\end{align}
for all $x=(x_1,x_2,0)\in\Gamma$ and all $t>t_0$.

\begin{remark}[Why the electric boundary conditions are not enough]\label{rem:electric_equations_do_not_close}

The phase vectors have already been determined by the temporal and spatial Snell laws, up to the choice of the physical signs of the normal components at the spatial interface. Thus the remaining unknowns are the six vector amplitudes.
Equivalently, there are $6\times 3=18$ scalar unknowns; in the non-degenerate case $\omega_2\neq\omega_3$, the electric boundary conditions alone provide at most $3+3+4+2=12$ scalar algebraic equations, so they cannot close the amplitude system.

The boundary conditions for the electric field give necessary relations among these amplitudes, but they do not by themselves determine all of them. Indeed, the temporal boundary condition in $\Omega_1$ gives only the equation
\[
\e_1^+(x,t_0^+)\left({\mathbf T}e^{-i\omega_3t_0}+{\mathbf R}e^{-i\omega_2t_0}\right)
=
\e^-(x,t_0^-){\mathbf I}e^{-i\omega_1t_0},
\qquad x\in\Omega_1,
\]
while the temporal boundary condition in $\Omega_2$ gives only that
\[
\e_2^+(x,t_0^+)\left({\mathbf T}_1e^{-i\omega_3t_0}+{\mathbf T}_2e^{-i\omega_2t_0}\right)
=
\e^-(x,t_0^-){\mathbf I}e^{-i\omega_1t_0},
\qquad x\in\Omega_2.
\]
These two equations determine sums of amplitudes, but not the individual amplitudes. In particular, the first equation does not separate ${\mathbf T}$ from ${\mathbf R}$, and the second equation does not separate ${\mathbf T}_1$ from ${\mathbf T}_2$.

At the spatial interface $\Gamma$, the tangential electric boundary condition
\begin{equation*}
\(\E_2^+(y,t)-\E_1^+(y,t)\)\times \mathbf{n}(y)=0 \quad \text{for each } y\in \Gamma \text{ and } t_0<t<b,
\end{equation*}
 gives two scalar equations. After the phase matching already obtained, these equations relate the tangential components of the amplitudes on the two sides of $\Gamma$. Similarly, the normal displacement condition
\[
\left(\e_2^+(y,t)\E_2^+(y,t)-\e_1^+(y,t)\E_1^+(y,t)\right)\cdot {\mathbf n}(y)=0,
\qquad y\in\Gamma,\ t>t_0,
\]
gives an additional scalar relation for the normal components. Hence the electric boundary conditions at $\Gamma$ impose relations between the fields on the two sides of the spatial interface, but they still do not close the full system for the six vector amplitudes.
\end{remark}

\subsection{Magnetic fields their boundary conditions}\label{subsec:magnetic_amplitude_equations}

The next step is to calculate the magnetic fields corresponding to each electric field component and then use the corresponding boundary conditions proved in Theorem \ref{thm:temporal boundary conditions}. The calculations below use \eqref{eq:faraday_minus} in the incident region, \eqref{eq:faraday_plus_Omega1} in $\Omega_1$ after the temporal jump, and \eqref{eq:faraday_plus_Omega2} in $\Omega_2$ after the temporal jump. Note that, in a homogeneous region where the material parameters are constant (or are evaluated at the relevant interface value), a plane-wave electric field of the form
\[
\E_A(x,t)={\mathbf A}\,e^{i\omega\left(\frac{{\mathbf q}\cdot x}{v_j}-t\right)}
\]
has an associated magnetic field obtained from Faraday's law,
\[
\nabla\times \E_A=-\frac{1}{c}\frac{\partial}{\partial t}(\mu_j\H_A),
\]
namely
\[
\H_A(x,t)=-\frac{c}{\mu_jv_j}\,{\mathbf A}\times {\mathbf q}\;e^{i\omega\left(\frac{{\mathbf q}\cdot x}{v_j}-t\right)},
\]
up to an integration term depending only on $x$, which must be fixed by the physical assumptions. 

Let us calculate first the magnetic fields. 

{\bf Stage 1}
Using the incident-region equation \eqref{eq:faraday_minus} when $t<t_0$, seek $\mathbf{H}_i$ satisfying
\[\nabla \times \E_i = -\dfrac{\mu_-}{c} \dfrac{\partial \mathbf{H}_i}{\partial t}.\]
Since
\[\E_i(x,t)= {\mathbf{I}}\, e^{i\omega_1 \left(\frac{\k_i \cdot x}{v_-}-t\right)}\]
we see that
\[\nabla \times \E_i =-i\omega_1 \left( {\mathbf{I}}\, \times \dfrac{\k_i}{v_-}\right) e^{i\omega_1 \left( \frac{\k_i\cdot x}{v_-}-t\right)}\]
and so integrating in time yields
\begin{align}\label{Hi_calculated}
\mathbf{H}_i = -\dfrac{c}{\mu_-}\int \nabla \times \E_i \,dt=-\dfrac{c}{\mu_-}\E_i \times \dfrac{\k_i}{v_-}
\end{align}
plus a field depending only on $x$ which we assume to be zero.

{\bf Stage 2}
For the temporally generated fields in $\Omega_1$, we appeal to \eqref{eq:faraday_plus_Omega1}. Suppose $t>t_0$. 
Since
\[\E_r(x,t)={\mathbf R}\,e^{i\omega_2 \left(\frac{\k_r^1\cdot x}{v_+}-t\right)}\]
we similarly find that
\begin{align}\label{Hr_calculated}
\mathbf{H}_r = -\dfrac{c}{\mu_+} \E_r \times \dfrac{\k_r^1}{v_+}
\end{align}
plus a field depending only on $x$ which we also assume to be zero.
Finally, since
\[\E_t(x,t) = {\mathbf T}\, e^{i\omega_3 \left(\frac{\k_t^1\cdot x}{v_+}-t\right)}\]
we find that
\begin{align}\label{Ht_calculated}
\mathbf{H}_t = -\dfrac{c}{\mu_+} \E_t \times \dfrac{\k_t^1}{v_+}.
\end{align}

{\bf Stage 3} Suppose $t>t_0$ and $x\in \Omega_1$.
Recall that incident into $\Gamma$ from $\Omega_1$ is
\[\tilde{\E_i}=\E_t+\E_r.\]
generating a reflected field in $\Omega_1$ of the form
\[\tilde{\E_r}(x,t)={\mathbf R}_1e^{i\omega_3\left(\frac{\m_r^1\cdot x}{v_+^1(t)}-t\right)}+{\mathbf R}_2e^{i\omega_2\left(\frac{\m_r^2\cdot x}{v_+^1(t)}-t\right)}\]
and a transmitted field in $\Omega_2$ of the form
\[\tilde{\E_\t}(x,t)={\mathbf T}_1e^{i\omega_3\left(\frac{\m_\t^1\cdot x}{v_+^2(t)}-t\right)}+{\mathbf T}_2e^{i\omega_2\left(\frac{\m_\t^2\cdot x}{v_+^2(t)}-t\right)}.\]
Using the same calculation as in Stage 2, with \eqref{eq:faraday_plus_Omega1} for the reflected field in $\Omega_1$ and \eqref{eq:faraday_plus_Omega2} for the transmitted field in $\Omega_2$, and taking the integration terms depending only on $x$ to be zero, the corresponding magnetic fields are

\[
\begin{aligned}
\tilde{\H_r}(x,t)
&=-\frac{c}{\mu_1^+}\left[
\frac{{\mathbf R}_1\times \m_r^1}{v_+^1(t)}e^{i\omega_3\left(\frac{\m_r^1\cdot x}{v_+^1(t)}-t\right)}
+\frac{{\mathbf R}_2\times \m_r^2}{v_+^1(t)}e^{i\omega_2\left(\frac{\m_r^2\cdot x}{v_+^1(t)}-t\right)}
\right],\\
\tilde{\H_\t}(x,t)
&=-\frac{c}{\mu_2^+}\left[
\frac{{\mathbf T}_1\times \m_\t^1}{v_+^2(t)}e^{i\omega_3\left(\frac{\m_\t^1\cdot x}{v_+^2(t)}-t\right)}
+\frac{{\mathbf T}_2\times \m_\t^2}{v_+^2(t)}e^{i\omega_2\left(\frac{\m_\t^2\cdot x}{v_+^2(t)}-t\right)}
\right].
\end{aligned}
\]
Here $\mu_1^+$ and $\mu_2^+$ denote the corresponding permeability values in $\Omega_1$ and $\Omega_2$, or their values at the relevant interface when the coefficients are evaluated there.
Thus, for the spatial boundary conditions after the interaction with $\Gamma$, we will take
\[
\H_1^+=\H_t+\H_r+\tilde{\H_r},
\qquad
\H_2^+=\tilde{\H_\t}.
\]
Once the magnetic fields are written in terms of the same electric field amplitudes, we can impose the temporal magnetic field boundary conditions (\ref{first_new_temporal_bc_forH}) and (\ref{second_new_temporal_bc_forH})
as well as the spatial magnetic field boundary conditions (\ref{eq:space boundary condition})
and (\ref{eq:bdry condition for H with dot product}).
 
 Using the temporal magnetic boundary condition \eqref{first_new_temporal_bc_forH} in $\Omega_1$, we obtain
\[
\begin{aligned}
0
&=\mu_1^+(x,t_0^+)(\H_t+\H_r)(x,t_0^+)-\mu^-(x,t_0^-)\H_i(x,t_0^-)\\
&=-c\left[
\frac{{\mathbf T}\times \k_t^1}{v_+^1(t_0)}e^{i\omega_3\left(\frac{\k_t^1\cdot x}{v_+^1(t_0)}-t_0\right)}
+\frac{{\mathbf R}\times \k_r^1}{v_+^1(t_0)}e^{i\omega_2\left(\frac{\k_r^1\cdot x}{v_+^1(t_0)}-t_0\right)}
\right]\\
&\qquad
+c\frac{{\mathbf I}\times \k_i}{v_-(t_0)}e^{i\omega_1\left(\frac{\k_i\cdot x}{v_-(t_0)}-t_0\right)}.
\end{aligned}
\]
The temporal phase matching gives
\[
\omega_3\frac{\k_t^1}{v_+^1(t_0)}
=\omega_2\frac{\k_r^1}{v_+^1(t_0)}
=\omega_1\frac{\k_i}{v_-(t_0)},
\]
so the common spatial exponential can be factored out. 


Thus, at the temporal interface, after using the phase matching in the spatial variables, the magnetic field boundary conditions \eqref{first_new_temporal_bc_forH} and \eqref{second_new_temporal_bc_forH} become
\begin{align*}
\frac{{\mathbf T}\times \k_t^1}{v_+^1(t_0)}e^{-i\omega_3t_0}
+\frac{{\mathbf R}\times \k_r^1}{v_+^1(t_0)}e^{-i\omega_2t_0}
&=
\frac{{\mathbf I}\times \k_i}{v_-(t_0)}e^{-i\omega_1t_0},
\qquad x\in\Omega_1,\\
\frac{{\mathbf T}_1\times \m_\t^1}{v_+^2(t_0)}e^{-i\omega_3t_0}
+\frac{{\mathbf T}_2\times \m_\t^2}{v_+^2(t_0)}e^{-i\omega_2t_0}
&=
\frac{{\mathbf I}\times \k_i}{v_-(t_0)}e^{-i\omega_1t_0},
\qquad x\in\Omega_2.
\end{align*}
At the spatial interface $\Gamma$, substitute $\H_2^+=\tilde{\H_\t}$ and $\H_1^+=\H_t+\H_r+\tilde{\H_r}$ into (\ref{eq:space boundary condition}). After using the spatial phase matching on $\Gamma$, all terms have the same tangential spatial exponential, while the time factors are either $e^{-i\omega_3t}$ or $e^{-i\omega_2t}$. Dividing out the common nonzero spatial factor and the common factor $-c$, we obtain
\[
\begin{aligned}
0
&=\left[
\left(
\frac{{\mathbf T}_1\times \m_\t^1}{\mu_2^+(y,t)v_+^2(t)}
-\frac{{\mathbf T}\times \k_t^1}{\mu_1^+(y,t)v_+^1(t)}
-\frac{{\mathbf R}_1\times \m_r^1}{\mu_1^+(y,t)v_+^1(t)}
\right)e^{-i\omega_3t}\right.\\
&\qquad\left.
+\left(
\frac{{\mathbf T}_2\times \m_\t^2}{\mu_2^+(y,t)v_+^2(t)}
-\frac{{\mathbf R}\times \k_r^1}{\mu_1^+(y,t)v_+^1(t)}
-\frac{{\mathbf R}_2\times \m_r^2}{\mu_1^+(y,t)v_+^1(t)}
\right)e^{-i\omega_2t}
\right]\times {\mathbf n}(y).
\end{aligned}
\]
In the non-degenerate case $\omega_2\neq\omega_3$, the two time exponentials are linearly independent on any interval of times, so their coefficients must vanish separately. Therefore the tangential magnetic field boundary condition \eqref{eq:space boundary condition} gives
\begin{align*}
\left(
\frac{{\mathbf T}_1\times \m_\t^1}{\mu_2^+(y,t)v_+^2(t)}
-\frac{{\mathbf T}\times \k_t^1}{\mu_1^+(y,t)v_+^1(t)}
-\frac{{\mathbf R}_1\times \m_r^1}{\mu_1^+(y,t)v_+^1(t)}
\right)\times {\mathbf n}(y)&=0,\\
\left(
\frac{{\mathbf T}_2\times \m_\t^2}{\mu_2^+(y,t)v_+^2(t)}
-\frac{{\mathbf R}\times \k_r^1}{\mu_1^+(y,t)v_+^1(t)}
-\frac{{\mathbf R}_2\times \m_r^2}{\mu_1^+(y,t)v_+^1(t)}
\right)\times {\mathbf n}(y)&=0,
\end{align*}
for $y\in\Gamma$ and $t>t_0$. Similarly, in the normal magnetic flux condition (\ref{eq:bdry condition for H with dot product}),the factors $\mu_2^+$ and $\mu_1^+$ multiply the magnetic fields and cancel the permeability denominators in the formulas for $\H_2^+$ and $\H_1^+$. After the same phase cancellation and separation of the two time exponentials, the normal magnetic flux condition \eqref{eq:bdry condition for H with dot product} gives
\begin{align*}
\left(
\frac{{\mathbf T}_1\times \m_\t^1}{v_+^2(t)}
-\frac{{\mathbf T}\times \k_t^1}{v_+^1(t)}
-\frac{{\mathbf R}_1\times \m_r^1}{v_+^1(t)}
\right)\cdot {\mathbf n}(y)&=0,\\
\left(
\frac{{\mathbf T}_2\times \m_\t^2}{v_+^2(t)}
-\frac{{\mathbf R}\times \k_r^1}{v_+^1(t)}
-\frac{{\mathbf R}_2\times \m_r^2}{v_+^1(t)}
\right)\cdot {\mathbf n}(y)&=0,
\end{align*}
for $y\in\Gamma$ and $t>t_0$. If $\omega_2=\omega_3$, the corresponding $\omega_2$ and $\omega_3$ terms must instead be combined before separating amplitude equations.
These equations provide the additional relations that are missing from the electric boundary conditions alone.

Finally, the source-free divergence equations \eqref{eq:divergence_free_electric_displacement} and \eqref{eq:divergence_free_magnetic_induction} should also be imposed on each plane-wave component. Equivalently, these are Maxwell's equations \eqref{divergence E zero} and \eqref{divergence B zero} with $\rho=0$, written through the constitutive relations \eqref{eq:constitutive} as $\nabla\cdot(\e\E)=0$ and $\nabla\cdot(\mu\H)=0$. 

For example, if $\E_A={\mathbf A}e^{i\omega({\mathbf q}\cdot x/v_j-t)}$ and $\e_j$ is independent of $x$, then
\[
\nabla\cdot(\e_j\E_A)
=i\e_j\frac{\omega}{v_j}({\mathbf q}\cdot{\mathbf A})e^{i\omega({\mathbf q}\cdot x/v_j-t)},
\]
so the divergence equation forces ${\mathbf q}\cdot{\mathbf A}=0$. Applying this to each electric plane-wave component gives the transversality conditions
\[
\k_i\cdot {\mathbf I}=0,
\qquad
\k_\t^1\cdot {\mathbf T}=0,
\qquad
\k_r^1\cdot {\mathbf R}=0,
\qquad
\m_\t^1\cdot {\mathbf T}_1=0,
\qquad
\m_\t^2\cdot {\mathbf T}_2=0,
\qquad
\m_r^1\cdot {\mathbf R}_1=0,
\qquad
\m_r^2\cdot {\mathbf R}_2=0.
\]
The first condition is a compatibility condition on the prescribed incident field, while the remaining conditions are additional equations for the unknown amplitudes.

\subsection{Assembly of the full linear amplitude system}\label{subsec:full_amplitude_system}

Collecting all of the preceding relations, in the non-degenerate case $\omega_2\neq\omega_3$ the amplitudes must solve the following linear system. 


\begin{align*}
\intertext{Temporal electric field equations from \eqref{first_new_temporal_bc} and \eqref{second_new_temporal_bc}:}
\e_1^+(x,t_0^+)\left({\mathbf T}e^{-i\omega_3t_0}+{\mathbf R}e^{-i\omega_2t_0}\right)
&=\e^-(x,t_0^-){\mathbf I}e^{-i\omega_1t_0}, \qquad x \in \Omega_1\\
\e_2^+(x,t_0^+)\left({\mathbf T}_1e^{-i\omega_3t_0}+{\mathbf T}_2e^{-i\omega_2t_0}\right)
&=\e^-(x,t_0^-){\mathbf I}e^{-i\omega_1t_0},\qquad x\in \Omega_2.
\intertext{Tangential electric field equations from \eqref{eq:space boundary condition for E}: for $y\in\Gamma$ and $t>t_0$,}
\left({\mathbf T}_1-{\mathbf T}-{\mathbf R}_1\right)\times {\mathbf n}(y)&=0,\\
\left({\mathbf T}_2-{\mathbf R}-{\mathbf R}_2\right)\times {\mathbf n}(y)&=0,
\intertext{Normal electric field equations from \eqref{eq:bdry condition for E with dot product}: for $y\in\Gamma$ and $t>t_0$,}
\left(\e_2^+(y,t){\mathbf T}_1-\e_1^+(y,t)\left({\mathbf T}+{\mathbf R}_1\right)\right)\cdot {\mathbf n}(y)&=0,\\
\left(\e_2^+(y,t){\mathbf T}_2-\e_1^+(y,t)\left({\mathbf R}+{\mathbf R}_2\right)\right)\cdot {\mathbf n}(y)&=0,
\intertext{Temporal magnetic field equations from \eqref{first_new_temporal_bc_forH} and \eqref{second_new_temporal_bc_forH}:}
\frac{{\mathbf T}\times \k_t^1}{v_+^1(t_0)}e^{-i\omega_3t_0}
+\frac{{\mathbf R}\times \k_r^1}{v_+^1(t_0)}e^{-i\omega_2t_0}
&=\frac{{\mathbf I}\times \k_i}{v_-(t_0)}e^{-i\omega_1t_0}, \qquad x \in \Omega_1\\
\frac{{\mathbf T}_1\times \m_\t^1}{v_+^2(t_0)}e^{-i\omega_3t_0}
+\frac{{\mathbf T}_2\times \m_\t^2}{v_+^2(t_0)}e^{-i\omega_2t_0}
&=\frac{{\mathbf I}\times \k_i}{v_-(t_0)}e^{-i\omega_1t_0},\qquad x\in\Omega_2.
\intertext{Tangential magnetic field equations from \eqref{eq:space boundary condition}: for $y\in\Gamma$ and $t>t_0$,}
\left(
\frac{{\mathbf T}_1\times \m_\t^1}{\mu_2^+(y,t)v_+^2(t)}
-\frac{{\mathbf T}\times \k_t^1}{\mu_1^+(y,t)v_+^1(t)}
-\frac{{\mathbf R}_1\times \m_r^1}{\mu_1^+(y,t)v_+^1(t)}
\right)\times {\mathbf n}(y)&=0,\\
\left(
\frac{{\mathbf T}_2\times \m_\t^2}{\mu_2^+(y,t)v_+^2(t)}
-\frac{{\mathbf R}\times \k_r^1}{\mu_1^+(y,t)v_+^1(t)}
-\frac{{\mathbf R}_2\times \m_r^2}{\mu_1^+(y,t)v_+^1(t)}
\right)\times {\mathbf n}(y)&=0.
\intertext{Normal magnetic flux equations from \eqref{eq:bdry condition for H with dot product}: for $y\in\Gamma$ and $t>t_0$,}
\left(
\frac{{\mathbf T}_1\times \m_\t^1}{v_+^2(t)}
-\frac{{\mathbf T}\times \k_t^1}{v_+^1(t)}
-\frac{{\mathbf R}_1\times \m_r^1}{v_+^1(t)}
\right)\cdot {\mathbf n}(y)&=0,\\
\left(
\frac{{\mathbf T}_2\times \m_\t^2}{v_+^2(t)}
-\frac{{\mathbf R}\times \k_r^1}{v_+^1(t)}
-\frac{{\mathbf R}_2\times \m_r^2}{v_+^1(t)}
\right)\cdot {\mathbf n}(y)&=0.
\end{align*}
\[
\begin{gathered}
\text{Transversality conditions from \eqref{eq:divergence_free_electric_displacement}
and \eqref{eq:divergence_free_magnetic_induction}:}\\
\k_i\cdot {\mathbf I}=\k_\t^1\cdot {\mathbf T}=\k_r^1\cdot {\mathbf R}=0,\\
\m_\t^1\cdot {\mathbf T}_1=\m_\t^2\cdot {\mathbf T}_2
=\m_r^1\cdot {\mathbf R}_1=\m_r^2\cdot {\mathbf R}_2=0.
\end{gathered}
\]

Equivalently, writing
\[
\begin{gathered}
{\mathbf I}=(I_1,I_2,I_3),\qquad
{\mathbf T}=(T_1,T_2,T_3),\qquad
{\mathbf R}=(R_1,R_2,R_3),\\
{\mathbf T}_1=(T_1^1,T_1^2,T_1^3),\qquad
{\mathbf T}_2=(T_2^1,T_2^2,T_2^3),\qquad
{\mathbf R}_1=(R_1^1,R_1^2,R_1^3),\qquad
{\mathbf R}_2=(R_2^1,R_2^2,R_2^3),\\
\k_i=(k_{i,1},k_{i,2},k_{i,3}),\qquad
\k_\t^1=(k_{\t,1}^1,k_{\t,2}^1,k_{\t,3}^1),\qquad
\k_r^1=(k_{r,1}^1,k_{r,2}^1,k_{r,3}^1),\\
\m_\t^a=(m_{\t,1}^a,m_{\t,2}^a,m_{\t,3}^a),\qquad
\m_r^a=(m_{r,1}^a,m_{r,2}^a,m_{r,3}^a),\quad a=1,2,
\end{gathered}
\]
and setting $\eta_q=e^{-i\omega_qt_0}$ for $q=1,2,3$, the temporal electric field equations are, for $j=1,2,3$,
\begin{align*}
\e_1^+(x,t_0^+)\left(T_j\eta_3+R_j\eta_2\right)
&=\e^-(x,t_0^-)I_j\eta_1,\\
\e_2^+(x,t_0^+)\left(T_1^j\eta_3+T_2^j\eta_2\right)
&=\e^-(x,t_0^-)I_j\eta_1.
\end{align*}
Since ${\mathbf n}(y)=(0,0,1)$, the tangential electric field equations reduce to
\begin{align*}
T_1^1-T_1-R_1^1&=0,
&T_1^2-T_2-R_1^2&=0,\\
T_2^1-R_1-R_2^1&=0,
&T_2^2-R_2-R_2^2&=0,
\end{align*}
and the normal electric field equations reduce to
\begin{align*}
\e_2^+(y,t)T_1^3-\e_1^+(y,t)\left(T_3+R_1^3\right)&=0,\\
\e_2^+(y,t)T_2^3-\e_1^+(y,t)\left(R_3+R_2^3\right)&=0.
\end{align*}
The temporal magnetic field equations are, again, written componentwise for the three cyclic triples $(j,\ell,m)=(1,2,3),(2,3,1),(3,1,2)$,
\begin{align*}
\frac{T_\ell k_{\t,m}^1-T_m k_{\t,\ell}^1}{v_+^1(t_0)}\eta_3
+\frac{R_\ell k_{r,m}^1-R_m k_{r,\ell}^1}{v_+^1(t_0)}\eta_2
&=\frac{I_\ell k_{i,m}-I_m k_{i,\ell}}{v_-(t_0)}\eta_1,\\
\frac{T_1^\ell m_{\t,m}^1-T_1^m m_{\t,\ell}^1}{v_+^2(t_0)}\eta_3
+\frac{T_2^\ell m_{\t,m}^2-T_2^m m_{\t,\ell}^2}{v_+^2(t_0)}\eta_2
&=\frac{I_\ell k_{i,m}-I_m k_{i,\ell}}{v_-(t_0)}\eta_1.
\end{align*}
For the spatial magnetic field equations, define, for the same three cyclic triples $(j,\ell,m)$,
\begin{align*}
\mathcal B_j^{(1)}
&=\frac{T_1^\ell m_{\t,m}^1-T_1^m m_{\t,\ell}^1}{\mu_2^+(y,t)v_+^2(t)}
-\frac{T_\ell k_{\t,m}^1-T_m k_{\t,\ell}^1}{\mu_1^+(y,t)v_+^1(t)}
-\frac{R_1^\ell m_{r,m}^1-R_1^m m_{r,\ell}^1}{\mu_1^+(y,t)v_+^1(t)},\\
\mathcal B_j^{(2)}
&=\frac{T_2^\ell m_{\t,m}^2-T_2^m m_{\t,\ell}^2}{\mu_2^+(y,t)v_+^2(t)}
-\frac{R_\ell k_{r,m}^1-R_m k_{r,\ell}^1}{\mu_1^+(y,t)v_+^1(t)}
-\frac{R_2^\ell m_{r,m}^2-R_2^m m_{r,\ell}^2}{\mu_1^+(y,t)v_+^1(t)},\\
\mathcal C_j^{(1)}
&=\frac{T_1^\ell m_{\t,m}^1-T_1^m m_{\t,\ell}^1}{v_+^2(t)}
-\frac{T_\ell k_{\t,m}^1-T_m k_{\t,\ell}^1}{v_+^1(t)}
-\frac{R_1^\ell m_{r,m}^1-R_1^m m_{r,\ell}^1}{v_+^1(t)},\\
\mathcal C_j^{(2)}
&=\frac{T_2^\ell m_{\t,m}^2-T_2^m m_{\t,\ell}^2}{v_+^2(t)}
-\frac{R_\ell k_{r,m}^1-R_m k_{r,\ell}^1}{v_+^1(t)}
-\frac{R_2^\ell m_{r,m}^2-R_2^m m_{r,\ell}^2}{v_+^1(t)}.
\end{align*}
Then the tangential magnetic field equations can be written as
\begin{align*}
\mathcal B_1^{(1)}&=0,
&\mathcal B_2^{(1)}&=0,\\
\mathcal B_1^{(2)}&=0,
&\mathcal B_2^{(2)}&=0,
\end{align*}
and the normal magnetic field equations can be written as
\begin{align*}
\mathcal C_3^{(1)}=0,
\qquad
\mathcal C_3^{(2)}=0.
\end{align*}
Finally, the transversality equations are
\begin{align*}
\sum_{j=1}^3 k_{i,j}I_j&=0,
&\sum_{j=1}^3 k_{\t,j}^1T_j&=0,
&\sum_{j=1}^3 k_{r,j}^1R_j&=0,\\
\sum_{j=1}^3 m_{\t,j}^1T_1^j&=0,
&\sum_{j=1}^3 m_{\t,j}^2T_2^j&=0,
&\sum_{j=1}^3 m_{r,j}^1R_1^j&=0,
&\sum_{j=1}^3 m_{r,j}^2R_2^j&=0.
\end{align*}

Thus the unknown amplitude components can be ordered as
\[
\mathbf u=\left(
T_1,T_2,T_3,
R_1,R_2,R_3,
T_1^1,T_1^2,T_1^3,
T_2^1,T_2^2,T_2^3,
R_1^1,R_1^2,R_1^3,
R_2^1,R_2^2,R_2^3
\right)^T\in\C^{18}.
\]
The incident amplitude ${\mathbf I}$ is prescribed and must satisfy the compatibility condition $\sum_{j=1}^3 k_{i,j}I_j=0$. The remaining amplitude equations form a linear system
\begin{equation}\label{eq:amplitude-system}
\mathcal A\mathbf u=\mathbf f,
\end{equation}
where $\mathbf f$ depends only on the incident amplitude and the rows of this system are as follows. To shorten notation, set
\[
\begin{gathered}
\epsilon_1^0=\e_1^+(x,t_0^+),\qquad
\epsilon_2^0=\e_2^+(x,t_0^+),\qquad
\epsilon_-^0=\e^-(x,t_0^-),\\
\epsilon_1^\Gamma=\e_1^+(y,t),\qquad
\epsilon_2^\Gamma=\e_2^+(y,t),\\
a_1=\frac{\eta_3}{v_+^1(t_0)},\qquad
b_1=\frac{\eta_2}{v_+^1(t_0)},\qquad
a_2=\frac{\eta_3}{v_+^2(t_0)},\qquad
b_2=\frac{\eta_2}{v_+^2(t_0)},\qquad
a_-=\frac{\eta_1}{v_-(t_0)},\\
s_1=\frac{1}{\mu_1^+(y,t)v_+^1(t)},\qquad
s_2=\frac{1}{\mu_2^+(y,t)v_+^2(t)},
\qquad
r_1=\frac{1}{v_+^1(t)},\qquad
r_2=\frac{1}{v_+^2(t)}.
\end{gathered}
\]
To identify the matrix and right hand side, let $\mathrm{Id}_3$ be the $3\times 3$ identity matrix,
\[
\Pi=\begin{pmatrix}1&0&0\\0&1&0\end{pmatrix},
\qquad
\xi=\begin{pmatrix}0&0&1\end{pmatrix},
\]
and for any vector $p=(p_1,p_2,p_3)$ define
\[
C(p)=
\begin{pmatrix}
0&p_3&-p_2\\
-p_3&0&p_1\\
p_2&-p_1&0
\end{pmatrix},
\qquad
C(p)A=A\times p.
\]
With respect to the block decomposition
\[
\mathbf u=({\mathbf T},{\mathbf R},{\mathbf T}_1,{\mathbf T}_2,{\mathbf R}_1,{\mathbf R}_2)^T,
\]
where each entry is a $3$-vector, the coefficient matrix is the $30\times 18$ matrix
\[
\mathcal A=\begin{pmatrix}
\mathcal A_E\\
\mathcal A_{M,0}\\
\mathcal A_{M,\Gamma}\\
\mathcal A_D
\end{pmatrix},
\qquad
\mathbf f=\begin{pmatrix}
\mathbf f_E\\
\mathbf f_{M,0}\\
\mathbf 0_6\\
\mathbf 0_6
\end{pmatrix}.
\]
The block dimensions are
\[
\mathcal A_E\in\C^{12\times 18},
\qquad
\mathcal A_{M,0}\in\C^{6\times 18},
\qquad
\mathcal A_{M,\Gamma}\in\C^{6\times 18},
\qquad
\mathcal A_D\in\C^{6\times 18}.
\]
Thus $\mathcal A$ has $12+6+6+6=30$ rows and $18$ columns, while
\[
\mathbf f_E\in\C^{12},
\qquad
\mathbf f_{M,0}\in\C^6,
\qquad
\mathbf 0_6\in\C^6,
\]
so $\mathbf f\in\C^{30}$.
Here
\begin{equation}\label{eq:electric-amplitude-matrix-block}
\mathcal A_E=\begin{pmatrix}
\epsilon_1^0\eta_3\mathrm{Id}_3&\epsilon_1^0\eta_2\mathrm{Id}_3&0&0&0&0\\
0&0&\epsilon_2^0\eta_3\mathrm{Id}_3&\epsilon_2^0\eta_2\mathrm{Id}_3&0&0\\
-\Pi&0&\Pi&0&-\Pi&0\\
0&-\Pi&0&\Pi&0&-\Pi\\
-\epsilon_1^\Gamma\xi&0&\epsilon_2^\Gamma\xi&0&-\epsilon_1^\Gamma\xi&0\\
0&-\epsilon_1^\Gamma\xi&0&\epsilon_2^\Gamma\xi&0&-\epsilon_1^\Gamma\xi
\end{pmatrix},
\qquad
\mathbf f_E=\begin{pmatrix}
\epsilon_-^0\eta_1{\mathbf I}\\
\epsilon_-^0\eta_1{\mathbf I}\\
\mathbf 0_2\\
\mathbf 0_2\\
0\\
0
\end{pmatrix},
\end{equation}
\begin{equation}\label{eq:temporal-magnetic-amplitude-matrix-block}
\mathcal A_{M,0}=\begin{pmatrix}
a_1C(\k_\t^1)&b_1C(\k_r^1)&0&0&0&0\\
0&0&a_2C(\m_\t^1)&b_2C(\m_\t^2)&0&0
\end{pmatrix},
\qquad
\mathbf f_{M,0}=\begin{pmatrix}
a_-C(\k_i){\mathbf I}\\
a_-C(\k_i){\mathbf I}
\end{pmatrix},
\end{equation}
\begin{equation}\label{eq:spatial-magnetic-amplitude-matrix-block}
\mathcal A_{M,\Gamma}=\begin{pmatrix}
-s_1\Pi C(\k_\t^1)&0&s_2\Pi C(\m_\t^1)&0&-s_1\Pi C(\m_r^1)&0\\
0&-s_1\Pi C(\k_r^1)&0&s_2\Pi C(\m_\t^2)&0&-s_1\Pi C(\m_r^2)\\
-r_1\xi C(\k_\t^1)&0&r_2\xi C(\m_\t^1)&0&-r_1\xi C(\m_r^1)&0\\
0&-r_1\xi C(\k_r^1)&0&r_2\xi C(\m_\t^2)&0&-r_1\xi C(\m_r^2)
\end{pmatrix},
\end{equation}
and
\begin{equation}\label{eq:transversality-amplitude-matrix-block}
\mathcal A_D=\begin{pmatrix}
(\k_\t^1)^T&0&0&0&0&0\\
0&(\k_r^1)^T&0&0&0&0\\
0&0&(\m_\t^1)^T&0&0&0\\
0&0&0&(\m_\t^2)^T&0&0\\
0&0&0&0&(\m_r^1)^T&0\\
0&0&0&0&0&(\m_r^2)^T
\end{pmatrix}.
\end{equation}

\subsubsection*{Solvability and compatibility for \eqref{eq:amplitude-system}}

Because the material parameters are piecewise constant, each of the quantities appearing in the entries of $\mathcal A$ and $\mathbf f$ is a fixed number after the corresponding side of the interface has been chosen. Thus, once the wave vectors and the incident amplitude are prescribed, $\mathcal A$ is a fixed $30\times 18$ matrix and $\mathbf f$ is a fixed vector in $\C^{30}$. The existence of transmitted and reflected amplitudes means that there is some vector $\mathbf u\in \C^{18}$ satisfying $\mathcal A\mathbf u=\mathbf f$. Equivalently, $\mathbf f$ must lie in the range, or column space, of $\mathcal A$:
\[
\mathbf f\in \operatorname{Ran}(\mathcal A).
\]
That is, $\mathbf f$ must be a linear combination of the columns of $\mathcal A$. This is the same as the rank condition
\[
\operatorname{rank}(\mathcal A)=\operatorname{rank}\begin{pmatrix}\mathcal A&\mathbf f\end{pmatrix}.
\]
Indeed, adjoining $\mathbf f$ as one more column should not increase the rank; if it does increase the rank, then $\mathbf f$ is not generated by the columns of $\mathcal A$, and the amplitude system has no solution.

It is useful to note here that $\mathbf f$ depends only on the prescribed incident wave and on the fixed material constants. Thus $\mathbf f$ contains no unknown transmitted or reflected amplitudes. This does not automatically imply solvability, but it makes solvability a completely explicit finite-dimensional algebraic check: for the given incident wave, one only has to test whether this known vector $\mathbf f$ lies in the column space of the known matrix $\mathcal A$, or equivalently whether the rank condition above holds. If this test fails, then that incident wave cannot be matched by transmitted and reflected amplitudes satisfying all the imposed boundary conditions. If it holds, the solution set is one particular solution plus $\ker(\mathcal A)$; in particular, the amplitudes are unique precisely when $\operatorname{rank}(\mathcal A)=18$.

\subsection{Coordinate form of the amplitude system}\label{subsec:amplitude_coordinate_system}

For the convenience of the reader, we now write the equation \eqref{eq:amplitude-system}, namely $\mathcal A\mathbf u=\mathbf f$, in coordinates.
The following electric equations are the coordinate expansion of the matrix pair $(\mathcal A_E,\mathbf f_E)$ in \eqref{eq:electric-amplitude-matrix-block}: the first six equations come from its first two block rows, the next four from its third and fourth block rows, and the final two from its fifth and sixth block rows.
\begin{align*}
\epsilon_1^0\eta_3T_1+\epsilon_1^0\eta_2R_1&=\epsilon_-^0\eta_1I_1,\\
\epsilon_1^0\eta_3T_2+\epsilon_1^0\eta_2R_2&=\epsilon_-^0\eta_1I_2,\\
\epsilon_1^0\eta_3T_3+\epsilon_1^0\eta_2R_3&=\epsilon_-^0\eta_1I_3,\\
\epsilon_2^0\eta_3T_1^1+\epsilon_2^0\eta_2T_2^1&=\epsilon_-^0\eta_1I_1,\\
\epsilon_2^0\eta_3T_1^2+\epsilon_2^0\eta_2T_2^2&=\epsilon_-^0\eta_1I_2,\\
\epsilon_2^0\eta_3T_1^3+\epsilon_2^0\eta_2T_2^3&=\epsilon_-^0\eta_1I_3,\\
-T_1+T_1^1-R_1^1&=0,
&-T_2+T_1^2-R_1^2&=0,\\
-R_1+T_2^1-R_2^1&=0,
&-R_2+T_2^2-R_2^2&=0,\\
-\epsilon_1^\Gamma T_3+\epsilon_2^\Gamma T_1^3-\epsilon_1^\Gamma R_1^3&=0,\\
-\epsilon_1^\Gamma R_3+\epsilon_2^\Gamma T_2^3-\epsilon_1^\Gamma R_2^3&=0.
\end{align*}
The temporal magnetic equations are the coordinate expansion of $(\mathcal A_{M,0},\mathbf f_{M,0})$ in \eqref{eq:temporal-magnetic-amplitude-matrix-block}: the first three equations come from its first block row and the final three from its second block row.
\begin{align*}
a_1\left(T_2k_{\t,3}^1-T_3k_{\t,2}^1\right)
+b_1\left(R_2k_{r,3}^1-R_3k_{r,2}^1\right)
&=a_-\left(I_2k_{i,3}-I_3k_{i,2}\right),\\
a_1\left(T_3k_{\t,1}^1-T_1k_{\t,3}^1\right)
+b_1\left(R_3k_{r,1}^1-R_1k_{r,3}^1\right)
&=a_-\left(I_3k_{i,1}-I_1k_{i,3}\right),\\
a_1\left(T_1k_{\t,2}^1-T_2k_{\t,1}^1\right)
+b_1\left(R_1k_{r,2}^1-R_2k_{r,1}^1\right)
&=a_-\left(I_1k_{i,2}-I_2k_{i,1}\right),\\
a_2\left(T_1^2m_{\t,3}^1-T_1^3m_{\t,2}^1\right)
+b_2\left(T_2^2m_{\t,3}^2-T_2^3m_{\t,2}^2\right)
&=a_-\left(I_2k_{i,3}-I_3k_{i,2}\right),\\
a_2\left(T_1^3m_{\t,1}^1-T_1^1m_{\t,3}^1\right)
+b_2\left(T_2^3m_{\t,1}^2-T_2^1m_{\t,3}^2\right)
&=a_-\left(I_3k_{i,1}-I_1k_{i,3}\right),\\
a_2\left(T_1^1m_{\t,2}^1-T_1^2m_{\t,1}^1\right)
+b_2\left(T_2^1m_{\t,2}^2-T_2^2m_{\t,1}^2\right)
&=a_-\left(I_1k_{i,2}-I_2k_{i,1}\right).
\end{align*}
The spatial magnetic equations are the coordinate expansion of $\mathcal A_{M,\Gamma}$ in \eqref{eq:spatial-magnetic-amplitude-matrix-block}: the first four equations come from the first two block rows containing $\Pi$, and the final two come from the last two block rows containing $\xi$.
\begin{align*}
s_2\left(T_1^2m_{\t,3}^1-T_1^3m_{\t,2}^1\right)
-s_1\left(T_2k_{\t,3}^1-T_3k_{\t,2}^1\right)
-s_1\left(R_1^2m_{r,3}^1-R_1^3m_{r,2}^1\right)&=0,\\
s_2\left(T_1^3m_{\t,1}^1-T_1^1m_{\t,3}^1\right)
-s_1\left(T_3k_{\t,1}^1-T_1k_{\t,3}^1\right)
-s_1\left(R_1^3m_{r,1}^1-R_1^1m_{r,3}^1\right)&=0,\\
s_2\left(T_2^2m_{\t,3}^2-T_2^3m_{\t,2}^2\right)
-s_1\left(R_2k_{r,3}^1-R_3k_{r,2}^1\right)
-s_1\left(R_2^2m_{r,3}^2-R_2^3m_{r,2}^2\right)&=0,\\
s_2\left(T_2^3m_{\t,1}^2-T_2^1m_{\t,3}^2\right)
-s_1\left(R_3k_{r,1}^1-R_1k_{r,3}^1\right)
-s_1\left(R_2^3m_{r,1}^2-R_2^1m_{r,3}^2\right)&=0,\\
r_2\left(T_1^1m_{\t,2}^1-T_1^2m_{\t,1}^1\right)
-r_1\left(T_1k_{\t,2}^1-T_2k_{\t,1}^1\right)
-r_1\left(R_1^1m_{r,2}^1-R_1^2m_{r,1}^1\right)&=0,\\
r_2\left(T_2^1m_{\t,2}^2-T_2^2m_{\t,1}^2\right)
-r_1\left(R_1k_{r,2}^1-R_2k_{r,1}^1\right)
-r_1\left(R_2^1m_{r,2}^2-R_2^2m_{r,1}^2\right)&=0.
\end{align*}
Finally, the six transversality equations are the coordinate expansion, in row order, of $\mathcal A_D$ in \eqref{eq:transversality-amplitude-matrix-block}.
\begin{align*}
k_{\t,1}^1T_1+k_{\t,2}^1T_2+k_{\t,3}^1T_3&=0,\\
k_{r,1}^1R_1+k_{r,2}^1R_2+k_{r,3}^1R_3&=0,\\
m_{\t,1}^1T_1^1+m_{\t,2}^1T_1^2+m_{\t,3}^1T_1^3&=0,\\
m_{\t,1}^2T_2^1+m_{\t,2}^2T_2^2+m_{\t,3}^2T_2^3&=0,\\
m_{r,1}^1R_1^1+m_{r,2}^1R_1^2+m_{r,3}^1R_1^3&=0,\\
m_{r,1}^2R_2^1+m_{r,2}^2R_2^2+m_{r,3}^2R_2^3&=0.
\end{align*}
This is a linear system in the $18$ unknown components of $\mathbf u$; depending on the geometry and material parameters, some rows may be linearly dependent or compatibility conditions may be required.

\subsection{Example: Oblique incidence with different material constants and unit phase directions}\label{subsec:amplitude_example}

We illustrate the amplitude system \eqref{eq:amplitude-system} in a concrete oblique-incidence case. Let $t_0=0$ and take
\begin{equation}\label{eq:example-material-constants}
\epsilon_-^0=1,
\qquad
\epsilon_1^0=\epsilon_1^\Gamma=2,
\qquad
\epsilon_2^0=\epsilon_2^\Gamma=2,
\qquad
\mu^-=1,
\qquad
\mu_1^+=2,
\qquad
\mu_2^+=6.
\end{equation}
Then
\begin{equation}\label{eq:example-velocities}
v_-=1,
\qquad
v_+^1=\frac12,
\qquad
v_+^2=\frac{1}{2\sqrt3}.
\end{equation}

Choose $\omega_1>0$ and prescribe
\begin{equation}\label{eq:example-incident-data}
{\mathbf I}=E_0\mathbf e_2,
\qquad
\k_i=\frac{\sqrt3}{2}\mathbf e_1+\frac12\mathbf e_3,
\qquad
E_0\in\C.
\end{equation}
Then $|\k_i|=1$ and $\k_i\cdot{\mathbf I}=0$. Since $\k_\t^1$ and $\k_r^1$ are unit vectors, \eqref{eq:transmitted and reflected vectors in the first medium} and \eqref{eq:example-velocities} give, for the positive temporal frequencies,
\begin{equation}\label{eq:example-temporal-data}
\omega_2=\omega_3=\frac{\omega_1}{2},
\qquad
\k_\t^1=\k_r^1=\k_i.
\end{equation}
The change at $t=0$ is a genuine temporal material jump: the parameters change from $(\epsilon_-^0,\mu^-)=(1,1)$ to $(\epsilon_1^0,\mu_1^+)=(2,2)$, the wave speed changes from $v_-=1$ to $v_+^1=1/2$, and the temporal frequency changes from $\omega_1$ to $\omega_1/2$. The two positive-frequency branches in \eqref{eq:example-temporal-data} coincide because their phase directions are required to be unit vectors. Moreover, the temporal impedances are matched, since $\sqrt{\mu^-/\epsilon_-^0}=\sqrt{\mu_1^+/\epsilon_1^0}=1$; consequently, the explicit solution below is reflectionless at the temporal interface, with ${\mathbf R}=0$. We retain this simplifying temporal jump so that the example isolates the additional scattering caused by the nontrivial spatial jump from $(2,2)$ to $(2,6)$.
The tangential components at $\Gamma$ are fixed by \eqref{eq:system showing GSL}. Choosing the normal signs to obtain the desired outgoing directions gives
\begin{equation}\label{eq:example-unit-directions}
\begin{gathered}
\m_\t^1=\frac12\mathbf e_1+\frac{\sqrt3}{2}\mathbf e_3,
\qquad
\m_\t^2=\frac12\mathbf e_1-\frac{\sqrt3}{2}\mathbf e_3,\\
\m_r^1=\frac{\sqrt3}{2}\mathbf e_1+\frac12\mathbf e_3,
\qquad
\m_r^2=\frac{\sqrt3}{2}\mathbf e_1-\frac12\mathbf e_3.
\end{gathered}
\end{equation}
All four vectors in \eqref{eq:example-unit-directions} are unit vectors and satisfy the generalized Snell law tangentially at $\Gamma$. Figure \ref{fig:example-unit-directions} displays these directions and their tangential components.

\begin{figure}[htbp]
\centering
\begin{tikzpicture}[>=Stealth, scale=2.15, line cap=round, line join=round, every node/.style={font=\footnotesize,text=black}]
\node[font=\bfseries\small] at (0,1.48) {Unit directions at $\Gamma$};
\draw[->] (-1.25,0) -- (1.55,0) node[right] {$\mathbf e_1$};
\draw[->] (0,-1.25) -- (0,1.35) node[above] {$\mathbf e_3$};
\draw[dashed] (0,0) circle (1);
\node[below left] at (0,0) {$\Gamma$};
\draw[->, very thick, blue] (0,0) -- (0.866,0.5);
\draw[blue] (0.866,0.5) -- (1.08,0.67);
\node[anchor=west, align=left, text=black, fill=white, inner sep=1pt] at (1.1,0.72) {$\k_i=\k_\t^1=\k_r^1$\\$=\m_r^1$};
\draw[->, very thick, ForestGreen] (0,0) -- (0.5,0.866);
\node[anchor=south, text=black, fill=white, inner sep=1pt] at (0.5,1.06) {$\m_\t^1$};
\draw[->, very thick, ForestGreen] (0,0) -- (0.5,-0.866);
\node[anchor=north, text=black, fill=white, inner sep=1pt] at (0.5,-1.06) {$\m_\t^2$};
\draw[->, very thick, BurntOrange] (0,0) -- (0.866,-0.5);
\draw[BurntOrange] (0.866,-0.5) -- (1.08,-0.67);
\node[anchor=west, text=black, fill=white, inner sep=1pt] at (1.1,-0.72) {$\m_r^2$};
\draw[dotted, ForestGreen] (0.5,0.866) -- (0.5,0);
\draw[dotted, blue] (0.866,0.5) -- (0.866,0);
\node[below, fill=white, inner sep=0.5pt] at (0.5,-0.03) {$1/2$};
\node[above, fill=white, inner sep=0.5pt] at (0.866,0.03) {$\sqrt3/2$};
\node[align=center, fill=white, inner sep=1pt] at (0,-1.45) {dotted drops indicate tangential components};
\end{tikzpicture}
\caption{Unit phase directions in the $\mathbf e_1\mathbf e_3$-plane. The horizontal coordinate is tangential to $\Gamma$, while the vertical coordinate is normal to $\Gamma$.}
\label{fig:example-unit-directions}
\end{figure}
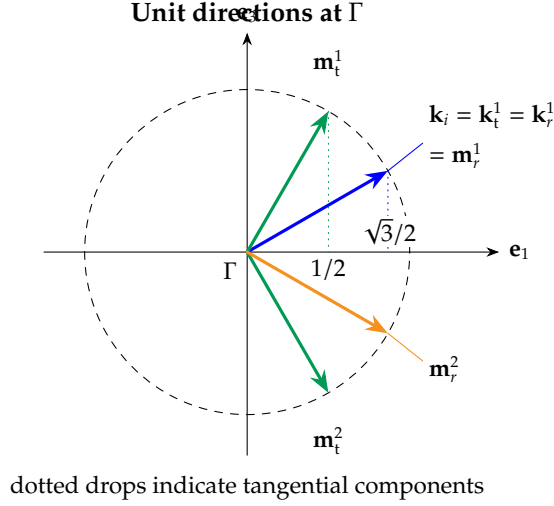

Since \eqref{eq:example-temporal-data} gives $\k_\t^1=\k_r^1=\k_i$, transversality puts ${\mathbf T}$ and ${\mathbf R}$ in $\k_i^\perp$. Use the orthonormal basis
\begin{equation}\label{eq:example-transverse-basis}
\mathbf e_2,
\qquad
\mathbf p_\perp=\frac12\mathbf e_1-\frac{\sqrt3}{2}\mathbf e_3
\end{equation}
for this plane, and write
\begin{equation}\label{eq:example-temporal-amplitudes}
{\mathbf T}=E_0\left(a\mathbf e_2+b\mathbf p_\perp\right),
\qquad
{\mathbf R}=E_0\left(c\mathbf e_2+d\mathbf p_\perp\right),
\qquad
a,b,c,d\in\C.
\end{equation}
The temporal electric boundary condition \eqref{first_new_temporal_bc} in $\Omega_1$ becomes
\begin{equation}\label{eq:example-temporal-electric-row}
2({\mathbf T}+{\mathbf R})=E_0\mathbf e_2,
\qquad\text{equivalently}\qquad
a+c=\frac12,
\quad
b+d=0.
\end{equation}
The temporal magnetic boundary condition \eqref{first_new_temporal_bc_forH} in $\Omega_1$ is obtained from the first identity in \eqref{eq:example-temporal-electric-row} by crossing with $\k_i$, so it gives no additional restriction. Hence the temporal interface alone leaves
\begin{equation}\label{eq:example-temporal-family}
{\mathbf T}=E_0\left(a\mathbf e_2+b\mathbf p_\perp\right),
\qquad
{\mathbf R}=E_0\left(\left(\frac12-a\right)\mathbf e_2-b\mathbf p_\perp\right),
\qquad
a,b\in\C.
\end{equation}

For the remaining branches, choose transverse unit vectors
\begin{equation}\label{eq:example-spatial-transverse-basis}
\mathbf p_\t^1=\frac{\sqrt3}{2}\mathbf e_1-\frac12\mathbf e_3,
\qquad
\mathbf p_\t^2=-\frac{\sqrt3}{2}\mathbf e_1-\frac12\mathbf e_3,
\qquad
\mathbf p_r^2=-\frac12\mathbf e_1-\frac{\sqrt3}{2}\mathbf e_3.
\end{equation}
Together with $\mathbf p_\perp$ for the $\m_r^1=\k_i$ branch, the most general transverse amplitudes are
\begin{equation}\label{eq:example-branch-amplitudes}
\begin{gathered}
{\mathbf T}_1=E_0(\alpha\mathbf e_2+\beta\mathbf p_\t^1),
\qquad
{\mathbf T}_2=E_0(\gamma\mathbf e_2+\delta\mathbf p_\t^2),\\
{\mathbf R}_1=E_0(\rho\mathbf e_2+\sigma\mathbf p_\perp),
\qquad
{\mathbf R}_2=E_0(\tau\mathbf e_2+\upsilon\mathbf p_r^2).
\end{gathered}
\end{equation}
Although \eqref{eq:example-temporal-data} is degenerate in frequency, we impose the separated branch equations from \eqref{eq:amplitude-system}; any solution then also satisfies the combined degenerate system. Substituting \eqref{eq:example-temporal-amplitudes} and \eqref{eq:example-branch-amplitudes} into \eqref{eq:amplitude-system} gives the following non-redundant scalar system:
\begin{equation}\label{eq:example-reduced-scalar-system}
\begin{gathered}
a+c=\frac12,
\qquad
b+d=0,
\qquad
\alpha+\gamma=\frac12,
\qquad
\beta=\delta=0,
\qquad
\alpha-\gamma=\frac16,\\
\rho=\alpha-a,
\qquad
\sigma=-b,
\qquad
\tau=\gamma-c,
\qquad
d=0,
\qquad
\upsilon=0,
\qquad
\tau=c+\gamma.
\end{gathered}
\end{equation}
Solving \eqref{eq:example-reduced-scalar-system} gives
\begin{equation}\label{eq:example-scalar-solution}
\begin{gathered}
a=\frac12,
\qquad
b=0,
\qquad
c=0,
\qquad
d=0,
\qquad
\alpha=\frac13,
\qquad
\beta=0,\\
\gamma=\frac16,
\qquad
\delta=0,
\qquad
\rho=-\frac16,
\qquad
\sigma=0,
\qquad
\tau=\frac16,
\qquad
\upsilon=0.
\end{gathered}
\end{equation}
Therefore the vector amplitudes are
\begin{equation}\label{eq:example-vector-amplitudes}
\begin{gathered}
{\mathbf T}=\frac{E_0}{2}\mathbf e_2,
\qquad
{\mathbf R}=0,
\qquad
{\mathbf T}_1=\frac{E_0}{3}\mathbf e_2,
\qquad
{\mathbf T}_2=\frac{E_0}{6}\mathbf e_2,\\
{\mathbf R}_1=-\frac{E_0}{6}\mathbf e_2,
\qquad
{\mathbf R}_2=\frac{E_0}{6}\mathbf e_2.
\end{gathered}
\end{equation}
Thus the solution is unique for the prescribed incident wave and chosen phase directions. 
\section{Conclusion}
We developed a distributional framework for deriving refraction laws in media with temporal and spatial interfaces. Starting from the Maxwell system in space-time distributions, we obtained the temporal jump conditions for the electric and magnetic fields under explicit trace assumptions on the fields and material parameters.

These boundary conditions were then used to derive generalized Snell laws for the temporal and spatial interfaces in the slab geometry. The resulting phase relations determine the admissible wave vectors, while the electric and magnetic field boundary conditions lead to a finite-dimensional linear system \eqref{eq:amplitude-system} for the six unknown amplitudes. The solvability of this system is expressed by the range, or equivalently rank, compatibility condition, showing that the incident amplitude and material constants cannot in general be chosen independently.

An explicit oblique-incidence example illustrates the construction. Finally, we expect that the same distributional methods can be used to treat configurations with multiple spatial interfaces, as well as more general spatio-temporal interface geometries and material laws.

\setcounter{equation}{0}

\bibliographystyle{unsrt}
\bibliography{timevaryingoptics_2.bib}


\end{document}